\documentclass[11pt]{amsart}
\usepackage{graphicx}
\usepackage{stmaryrd}
\usepackage{amssymb}
\usepackage{amsthm}
\usepackage{dsfont}
\usepackage{hyperref}
\usepackage{mathrsfs}
\usepackage{stmaryrd}
\usepackage[lite,initials,msc-links]{amsrefs}
\usepackage{amsmath}
\usepackage{centernot}
\usepackage{enumerate}
\usepackage{verbatim}
\usepackage{dsfont}
\usepackage{bm}
\usepackage{geometry}
\usepackage{faktor}
\usepackage{relsize}
\usepackage{centernot}
\usepackage{enumerate}
\usepackage{verbatim}
\usepackage{mathtools}
\usepackage{color,soul}
\usepackage{centernot}
\usepackage[mathscr]{euscript}
\usepackage{caption}

\newtheorem{theorem}{Theorem}[section]
\newtheorem{lemma}[theorem]{Lemma}
\newtheorem{proposition}[theorem]{Proposition}
\newtheorem{corollary}[theorem]{Corollary}
\numberwithin{equation}{section}

\theoremstyle{definition}

\theoremstyle{remark}
\newtheorem{remark}{Remark} 
\theoremstyle{question}

\newcommand{\n}{\mathbf{n}}

\newcommand{\s}{\textnormal{s}}

\newcommand{\con}[1]{ \xleftrightarrow{#1}}
\newcommand{\ncon}[1]{ \centernot{\xleftrightarrow{#1}}}

\title[Spins, percolation and height functions]{Spins, percolation and height functions}

\date{\today}

\author{Marcin Lis}{
\address{Faculty of Mathematics\\ University of Vienna \\
Oskar-Morgenstern-Platz 1\\
1090 Wien}
\email{marcin.lis@univie.ac.at}}

\begin{document}

\begin{abstract} 
To highlight certain similarities in combinatorial representations of several well known two-dimensional models of statistical mechanics, we introduce and study a new family of models which specializes to these cases after a proper tuning of the parameters.

To be precise, our model consists of two independent standard Potts models, with possibly different numbers of spins and different coupling constants (the four parameters of the model), 
defined jointly on a graph embedded in a surface and its dual graph, and conditioned
on the event that the primal and dual interfaces between spins of different value do not intersect. We also introduce naturally related height function and bond percolation models, and
we discuss their basic properties and mutual relationship.

As special cases we recover the standard Potts and random cluster model, the six-vertex model and loop $O(n)$ model,
the random current, double random current and XOR-Ising model. 

\end{abstract}

\maketitle
\section*{Introduction}

Combinatorial expansions are ubiquitous in statistical mechanics and a thorough understanding of their interplay often leads to a transfer of information between different models.
In an attempt to capture some of their common features,
we introduce a new family of two-dimensional models with four parameters $(q,q',a,b) \in \{1,2,\ldots \}^2\times (0,1]^2$.
To be precise, for a graph embedded in a surface, we consider three jointly coupled models (each of which comes in a primal and a dual version) in the form of
\begin{itemize}
\item spin models $(\sigma,\sigma')$ defined on the vertices and faces,
\item bond percolation models $(\omega,\omega')$ on the primal and dual edges, 
\item {height functions} $(h,h')$ defined on the vertices and faces. 
\end{itemize}
The starting point is the spin model which is given by a pair of independent primal and dual Potts models with $q$ and $q'$ spins, and coupling constants satisfying $a=e^{-J}$ and $b=e^{-J'}$ respectively, and conditioned on the event that their interfaces do not intersect.
The percolation model is then built on top of the spin model using additional randomness, and the height function is a deterministic function of the spin configuration.
We note that the relationship between the spin and percolation model 
is a generalization of the Edwards--Sokal coupling between Potts and Fortuin--Kasteleyn random cluster models~\cite{Potts,FK,EdwSok}.

Our model includes as special cases the
\begin{itemize}
\item $\textnormal{FK}(qq')$ random cluster model for $a+b=1$, 
\item staggered six-vertex model for $q=q'=2$,
\item loop $O(n)$ model for $q=n$, $q'=2$ and $b=1$,
\item random current model for $q=1$, $q'=2$ and $a^2+b^2=1$, 
\item double random current and XOR-Ising model for $q=q'=2$ and\\ {$a^2+b^2=1$}. 
\end{itemize}
As a result, we give a unified framework for some of the known relations between these models~\cite{Rys,AshTel,Weg,EdwSok,Nie,PfiVel,GriJan,Dub,BoudeT,LupWer,GlaPel,SpiRay}.

We also study the interplay between the three different instances of the model. For example we compare the variance of the height function evaluated at a face with the expected number of clusters of~the percolation model that surround that face. This is used, together with comparison arguments with the random cluster model, to prove that the height function on the square lattice is \emph{localized} 
meaning that its variance is uniformly bounded from above in the size of the system. We do this for two different regimes of parameters:
\begin{itemize}
\item when $a+b=1$, except when $q=q'=2$ and $a=b=1/2$. In the latter case it is shown that the height function \emph{delocalizes}, i.e., has unbounded variance as the system grows. 
This was first proved (for different boundary conditions) in~\cite{GlaPel},
\item when either $a$ or $b$ is small, i.e., $a< ({\sqrt{q'}+1})^{-1}$ or $b< ({\sqrt{q}+1})^{-1}$.
\end{itemize}
These considerations are motivated by and should be compared with the study of the behaviour of the height function of 
the six-vertex model on the square lattice~\cite{Pau,Lieb}. In this article we focus on the staggered version of the model since it fits more naturally into our framework. The staggered model and the standard one actually agree when $a=b$.
It has been proved that in this case the height function is localized for $a<1/2$~\cite{GlaPel}, and it is expected that it delocalizes for all $a\geq 1/2$. The only 
rigorously solved cases in this regime are so far $a=1/2$~\cite{GlaPel}, $a=\sqrt{2}/2$~\cite{Dub,Ken01} and its small neighbourhood~\cite{GMT}, and $a=1$~\cite{CPST}.

A related question in the setting of the more general model introduced here is to find parameters $(q,q',a,b)$ for which the height function on the square lattice delocalizes. A natural 
candidate seems to be (a subset of) the \emph{self-dual} line $q=q'$ and $a=b>1/2$. However, we were unable to prove any result to that effect. Some partial considerations are presented at the end of this article.
One of the major obstacles in the analysis is that the associated percolation model lacks positive association in this regime.

We note that in the special case $q=q'=2$ and $a=b\leq 1/2$
the coupled percolation models $(\omega,\omega')$ studied here were independently introduced by Ray and Spinka in the article \cite{SpiRay} which appeared during the preparation of this manuscript.
For $q=q'=2$, the laws of the marginals on $\omega$ and $\omega'$ are also present in the work of Glazman and Peled~\cite{GlaPel} on the six-vertex model,
and are closely related to the percolation models introduced by Pfister and Velenik~\cite{PfiVel}. 
We also note that the spin models $(\sigma,\sigma')$ are a variant of the interaction-round-a-face model studied in~\cite{BaxterBook,OwczBax}.

One of the novelties of our approach is to study the joint law of $(\omega,\omega',\sigma,\sigma')$ together with the associated height function. For instance, a useful feature of this coupling will be that for $a+b\geq 1$, the two percolation configurations $(\omega,\omega')$ are such that no edge and its dual edge are simultaneously open. This does not hold when $a+b\leq1$ which is the regime studied in~\cite{SpiRay}.

This article is organized as follows: 
\begin{itemize}
\item In Sect.~\ref{sec:model} we define the model and describe the basic relationship between the spin and percolation models.
\item
In Sec.~\ref{sec:other} we show how in special cases we recover other known models of statistical mechanics mentioned above.
\item
In Sect.~\ref{sec:unconstrained}, using duality arguments, we provide an alternative representation of the planar spin model as a classical unconditional spin model which turns out to be a special 
case of the $(N_\alpha,N_\beta)$ model of Domany and Riedel~\cite{DomRie}.
This relation is a generalization of the Ashkin-Teller model representation of the six-vertex model~\cite{Nie,Weg}.
\item
In Sect.~\ref{sec:posass} we review positive association of the spin and percolation model. We obtain it for $\omega$ and $\omega'$ when $a+b\leq 1$,
and for $\sigma$ and $\sigma'$ for all $a,b\in (0,1]$. These are generalizations of the results of~\cite{GlaPel,SpiRay}.
\item
In Sect.~\ref{sec:interplay} we study the interplay between the spin, percolation and height function model in more detail:
We prove the second Griffiths inequality for the spin model in the case when $a+b\leq 1$, and we compare the variance of the difference of the height function between two points with the expected number of clusters of the percolation model that disconnect these two points from each other. 
\item In Sect.~\ref{sec:square} we study the asymptotic behaviour the model defined on the square lattice.
\end{itemize}

{\bf Acknowledgements} I am grateful to Roland Bauerschmidt, Hugo Duminil-Copin and Aran Raoufi for the discussions on the double random current and XOR-Ising model that we had in 2017 at IHES, Bures-sur-Yvette, and that were the inspiration for this work. 
I also thank Alexander Glazman and Ron Peled for their very useful comments and suggestions, and Jacques H.H. Perk for bringing to my attention the model of Domany and Riedel~\cite{DomRie}. 

\section{The model} \label{sec:model}
The basis for our construction will be the Potts model. Let $Q$ be a finite set with $q$ elements.
Recall that for a finite graph $G=(V,E)$ and a coupling constant~$J$, the \emph{$q$-state Potts model}~\cite{Potts} is a probability measure on $Q^V$
given by
\begin{align} \label{def:Potts}
\mu(\s)= \frac1{Z_{G,q}} \exp\Big(-J \sum_{\{v_1,v_2\}\in E} \mathbf 1 \{\s(v_1)\neq \s(v_2)\} \Big), \quad \s \in Q^V,
\end{align}
where ${Z_{G,q}}$ is the \emph{partition function}. To denote the dependence on the parameter, we will write 
$Z_{G,q}=Z_{G,q}(x)$, where $x=e^{J}-1$.
We say that the model is \emph{ferromagnetic} if $J\geq 0$ (or equivalently $x\geq0$) and \emph{antiferromagnetic} if $J\leq 0$ ($-1\leq x\leq 0$).
We note that our definition is the standard one (see e.g.\ \cite{RC}) up to a rescaling of the weight in \eqref{def:Potts} by $e^{-J|E|}$. 

The $q$-state Potts model is directly related to the \emph{$\textnormal{FK}(q)$ random cluster model}~\cite{FK} by the classical \emph{Edwards--Sokal coupling}~\cite{EdwSok},
where for each edge $\{v_1,v_2\}$ satisfying $\s(v_1)=\s(v_2)$, one declares it open with probability $1-e^{-J}$ and independently of other edges. The resulting configuration of open edges $\zeta$
gives rise to a bond percolation model which is the random cluster model. Moreover, in this coupling, conditioned on~$\zeta$, the spins $\s$ can be recovered by choosing a uniform spin from $Q$ independently for each \emph{cluster} of $\zeta$, i.e., a connected component of $(V,\zeta)$, including isolated vertices.

We are now ready to define our model. 
\subsection{Spin model}
Let $M$ be a compact, orientable surface with no boundary, or the plane.
Let $\mathsf G=(\mathsf V, \mathsf E)$ be a finite connected graph embedded in $M$ in such a way that each face is a topological disc, and let $\mathsf G^*=(\mathsf U,\mathsf E^*)$ be its dual also embedded in $M$, where $\mathsf U$ is identified with the set of faces of $\mathsf G$.
For an edge $e\in \mathsf E\cup \mathsf E^*$, we write $e^*\in \mathsf E\cup \mathsf E^*$ for its dual edge. Similarly for $\omega \subseteq \mathsf E\cup \mathsf E^*$, we define $\omega^*=\{ e^*: e\in \omega\}$.

Fix $q,q'\in \{1,2,\ldots \}$ and choose two symmetric sets $ Q, Q' \subset \mathbb C$, i.e., satisfying $Q=-Q$ and $Q'=-Q'$, and such that $|Q|=q$ and $|Q'|=q'$. A \emph{spin configuration} on $\mathsf V$ (resp.\ $\mathsf U$) is any function $\sigma: \mathsf V \to Q$ (resp.\ $\sigma': \mathsf U \to Q'$).
We define the \emph{contour configurations} $\eta(\sigma) \subseteq \mathsf E^*$ of $\sigma$ to be the set of all dual edges $e^*$ such that the endpoints of the corresponding primal edge $e$ are assigned
different spin by~$\sigma$. We also define $\eta(\sigma')\subseteq \mathsf E$ in a dual fashion.
The configuration space of our \emph{(constrained) spin model} is
\begin{align} \label{def:sigma}
\Sigma&= \{(\sigma,\sigma') \in { Q}^{\mathsf V} \times { Q'}^{\mathsf U}: \eta(\sigma)^* \cap \eta(\sigma') =\emptyset \}. 
\end{align}
In other words, this is the set of all pairs of primal and dual spin configurations $(\sigma,\sigma')$ whose {interfaces}, interpreted as subsets of $M$, do not intersect.
Equivalently, 
\begin{align} \label{eq:divfree}
(\sigma(v_1)-\sigma(v_2))(\sigma'(u_1)-\sigma'(u_2))=0
\end{align}
for every pair of a primal edge $\{v_1,v_2\}$ and its dual $\{ u_1,u_2\}$.

We study a probability measure on $\Sigma$ given by
\begin{align} \label{eq:defspin}
\mathbf{P}(\sigma,\sigma') = \frac{1}{\mathcal Z}  a^{|\eta(\sigma')|} b^{|\eta(\sigma)|},
\end{align}
where $a,b\in(0,1]$ are parameters of the model, and $\mathcal Z=\mathcal Z(q,q',a,b)$ is the partition function. This measure is equivalent to a pair of independent primal and dual ferromagnetic Potts models 
with $q$ and $q'$ spins, with coupling constants $J=-\log b$ and $J'=-\log a$ respectively, and conditioned on~$\Sigma$.

From this definition we immediately get the following description of mutual conditional laws for $\sigma$ and $\sigma'$.

\begin{corollary} \label{cor:PP}
Conditioned on $\sigma'$, $\sigma$ is distributed like the $q$-state Potts model defined on the quotient graph $\mathsf {G}/{\eta(\sigma')}$ where each connected component of $\eta$ becomes a single vertex. By duality, the analogous statement holds true when the roles of $\sigma$ and $\sigma'$ are exchanged. 
\end{corollary}

\begin{remark}
We will only consider homogeneous weights $a$ and $b$, but most of our considerations generalize to non-homogeneous situations with $a_e = e^{- J'_{e^*}}$ and $b_{e^*} = e^{- J_{e}}$,
where $J$ and $J'$ are arbitrary sets of positive coupling constants on the primal and dual edges respectively. 
\end{remark}

\subsection{The height function}
Assume that $M$ is of genus zero. We say that $\{v_1,u_1,v_2,u_2\}$ is a \emph{quad}, if $\{v_1,v_2\}\in \mathsf E$ and $\{v_1,v_2\}^*=\{u_1,u_2\}$. 
For $(\sigma,\sigma')\in \Sigma$, we will consider a \emph{height function} $ H: \mathsf V \cup \mathsf U \to \mathbb{R}$ defined up to a constant by the rule: 
If $u\in \mathsf U$ and $v\in \mathsf V$ belong to the same quad, then 
\begin{align} \label{eq:Hdef}
 H(u)-  H(v) = \sigma(v)\sigma'(u).
\end{align}
The constant can be chosen by fixing the value of the function at a particular vertex or face. 
That these relations are consistent follows from condition \eqref{eq:divfree}. Indeed, \eqref{eq:divfree} is equivalent to the fact that the sum of the gradients~\eqref{eq:Hdef} around
each quad is zero.
We will denote by $h$ and $h'$ the restriction of $ H$ to $\mathsf V$ and $\mathsf U$ respectively.
Note that if $\{ v_1,v_2\}$ and $\{u_1,u_2\}$ are mutually dual edges, then
\begin{align}
h(v_2)-h(v_1) &= \sigma'(u_1)(\sigma(v_2)-\sigma(v_1))=\sigma'(u_2)(\sigma(v_2)-\sigma(v_1)),  \\
h'(u_2)-h'(u_1) &= \sigma(v_1)(\sigma'(u_2)-\sigma'(u_1))=\sigma(v_2)(\sigma'(u_2)-\sigma'(u_1)). \label{eq:proph2}
\end{align}

It follows from the definition that $h$ is constant on the clusters of constant spin $\sigma$, and $h'$ is constant on the clusters of constant spin $\sigma'$.

\begin{remark}
For surfaces of higher genus one can define in the same way a height function on the universal cover of $M$. Equivalently, one can talk about the increment of the height function between two points taken
along a curve, up to the homotopy of the curve. We will use the latter definition applied to the torus in Theorem~\ref{thm:selfdual}.
\end{remark}

\begin{remark}
So far we did not use the assumption that the sets $Q$ and $Q'$ are symmetric. This will be relevant for the considerations in Sect.~\ref{sec:interplay}.
A more general framework for the height function would be to consider arbitrary $Q,Q'$, and functions $f:Q\to \mathbb C$, $f':Q'\to \mathbb C$ such that $f$ and $f'$ are symmetric random variables 
when elements of $Q$ and $Q'$ are chosen uniformly. Then the gradient of the height function as in \eqref{eq:Hdef} could be given by $H(u)-  H(v) = f(\sigma(v))f'(\sigma'(u))$.
The slightly less general definition studied here corresponds to $f$ and $f'$ being injective. 
\end{remark}

\subsection{Bond percolation}
We also augment the model with a bond percolation configuration using the following procedure: 
Given $(\sigma,\sigma')\in \Sigma$ sampled according to $\mathbf{P}$,
\begin{enumerate}
\item \label{step1} Declare each primal edge in $\eta(\sigma')$ and each dual edge in $  \eta(\sigma)$ open. This is to say that
an edge is open if the corresponding dual edge carries two {different spins} in the dual spin configuration.
\item \label{step2} For each pair of a primal and its dual edge $e$ and $e^*$ such that neither $e\in\eta(\sigma')$ nor $e^*\in \eta(\sigma)$, and independently of other such pairs, declare the state of the edges with the following probabilities chosen depending on the value of $a+b$:
\begin{center}
\begin{tabular}{ |c|c|c| } 
  \cline{2-3}
  \multicolumn{1}{c|}{}&
  \multicolumn{1}{c|}{$\quad a+b\leq 1\quad $} &
  \multicolumn{1}{c|}{$\quad a+b\geq 1 \quad $ }  \\ 
 \hline
  $e$ open, $e^*$ closed & $a$ & $1-b$ \\ 
  $e$ closed, $e^*$ open & $b$ & $1-a$ \\ 
 both $e$, $e^*$ open & $1-a-b$ & $0$\\
  both $e$, $e^*$ closed & $0$ & $a+b-1$\\
 \hline
\end{tabular}
\end{center}
Note that in both cases the probability of opening $e$ and $e^*$ is $1-b$ and $1-a$ respectively.

\end{enumerate}

We call the resulting set of all open primal and dual edges $\omega$ and $\omega'$ respectively.
Note that $\omega\setminus \eta(\sigma')$ is exactly the set of open edges from the Edwards--Sokal coupling mentioned above 
applied to the Potts model $\sigma$ on the quotient graph $\mathsf {G}/{\eta(\sigma')}$ \footnote{We thank Ron Peled for this observation}. By duality, the same is true for $\omega'\setminus \eta(\sigma)$ and $\sigma'$.

A \emph{cluster} of $\omega$, resp.\ $\omega'$, is a connected component of the graph $(\mathsf V, \omega)$, resp.\ $(\mathsf U,\omega')$, including the isolated vertices.
We define
\begin{align*}
\Omega\Sigma = \{(\omega,\omega',\sigma,\sigma'): & \  \sigma \textnormal{ constant on clusters of } \omega \textnormal{ and } \eta(\sigma) \subseteq \omega',\\
	&\   \sigma' \textnormal{ constant on clusters of } \omega' \textnormal{ and } \eta(\sigma') \subseteq \omega\},
\end{align*}
where $(\sigma,\sigma') \in \Sigma$, to
be the space of consistent configurations for the spin model augmented with the sets of open edges, and
we denote by $\mathbf{P}(\omega,\omega',\sigma,\sigma')$ the probability measure on $\Omega\Sigma$ given by the coupling above.

\begin{remark}
We note that the left-hand side of the table above for $q=q'=2$ and $a=b$ describes the process studied by Ray and Spinka~\cite{SpiRay}.
\end{remark}

For $\xi\subseteq \mathsf E$, we write $\xi^{\dagger}=\mathsf E^*\setminus \xi^*$. It follows from the definition that 
\begin{align*}
\omega^{\dagger} \subseteq \omega'  \text{ for }a+b\leq1,  \qquad \text{ and } \qquad \omega^{\dagger}\supseteq \omega'  \text{ for } a+b\geq1
\end{align*}
almost surely.
We define $\Omega\Sigma_{\leq 1} $ and $\Omega\Sigma_{\geq 1}$ respectively to be $\Omega\Sigma$ with these additional restrictions imposed on $\omega$ and $\omega'$.
Note that for the boundary case $a+b=1$, we have that $\omega^{\dagger} = \omega'$ almost surely.

By definition, the weight of each configuration $(\omega,\omega',\sigma,\sigma')\in \Omega\Sigma_{\leq1}$ is
\begin{align} \label{eq:weight1}
a^{|\omega^* \setminus \omega'|} b^{|\omega' \setminus \omega^*|}(1-a-b)^{|\omega^* \cap \omega'|},
\end{align}
and the weight of $(\omega,\omega',\sigma,\sigma')\in \Omega\Sigma_{\geq1}$ is
\begin{align} \label{eq:weight2}
a^{|\eta(\sigma')|}(1-b)^{|\omega \setminus \eta(\sigma')|} b^{|\eta(\sigma)|}(1-a)^{|\omega' \setminus \eta(\sigma)|}(a+b-1)^{\mathsf E^* \setminus (\omega^* \cup \omega')}.
\end{align}
Note that \eqref{eq:weight1} is independent of $(\sigma,\sigma')$.

It turns out that the clusters of $\omega$ encode geometrically the stochastic dependencies of the spin model $\sigma$.
This is manifested in the following Edwards--Sokal property of the coupling between $\sigma$ and $\omega$. Clearly, the same holds for $\omega'$ and $\sigma'$ by duality.
In the following, for $\xi\subseteq \mathsf E$, we denote by $V(\xi)$ the set of vertices incident on at least one edge in $\xi$.

\begin{proposition} \label{prop:coupling} Conditioned on $\omega$, 
\begin{enumerate}
\item $\sigma$ is distributed like an independent uniform assignment of a spin from $Q$ to each cluster of $\omega$. \label{eq:step1}
\item $\sigma'$ is distributed like the $q'$-state Potts model with coupling constant $J$ satisfying $e^{-J} = \tfrac{a}{1-b}$, and defined on the dual $(\mathsf V(\omega), \omega)^*$ of $(\mathsf V(\omega), \omega)$, i.e.,
the graph whose vertices are the faces of $\omega$ and where two faces are adjacent, if they share an edge in $\omega$. Note that multiple edges and loops are possible.
\label{eq:step2}
\item in particular, $\sigma$ and $\sigma'$ are independent.  \label{eq:step3}
\end{enumerate}
\end{proposition}

\begin{proof}
We claim that for fixed $(\omega,\sigma')$ with $\eta(\sigma')\subseteq \omega$, the weight of each consistent configuration 
$(\omega,\sigma,\sigma')$, i.e., such that $\sigma$ is constant on the clusters of $\omega$, is equal to
\begin{align} \label{eq:condweight}
a^{|\eta(\sigma')|}(1-b)^{|\omega\setminus \eta(\sigma')|} b^{|\omega^{\dagger}|},
\end{align}
where $\omega^{\dagger}=\mathsf E^*\setminus \omega^*$,
and in particular is independent of $\sigma$.
Indeed this follows from the fact that each edge in 
\begin{itemize}
\item $\eta(\sigma) $ contributes weight $b$ by the definition of the spin model,
\item $\omega^{\dagger}\setminus \eta(\sigma)$ also contributes weight $b$ since this is the probability that a dual edge $\{ u_1,u_2\}$ with $\sigma'({u_1})= \sigma'({u_2})$ ends up in $\omega^{\dagger}$ in step \eqref{step2} of the definition of the edge percolation model.
\end{itemize}
This means that conditioned on $(\omega,\sigma')$, we have a uniform distribution on all spin configurations $\sigma$ such that $\eta(\sigma) \subseteq \omega^{\dagger} $. 
This is equivalent to choosing an independent uniform spin for each cluster of $\omega$ and we conclude property \eqref{eq:step1}.

Property \eqref{eq:step2} also follows from \eqref{eq:condweight}, the definition of the Potts model, and the fact that the only constraint on $\sigma'$ is that $\eta(\sigma')\subseteq \omega$.

Conditional independence of $\sigma$ and $\sigma'$ follows from the fact that \eqref{eq:condweight} does not depend on $\sigma$.
\end{proof}
We note that property ~\eqref{eq:step1} for $q=q'=2$ was first studied in~\cite{GlaPel, SpiRay}.

\begin{remark}
Since \eqref{eq:weight1} is independent of both $\sigma$ and $\sigma'$ for $a+b\leq1$, property \eqref{eq:step1} from the proposition above holds in this case 
simultaneously for $\sigma$ and $\sigma'$ when conditioned on $(\omega,\omega')$.
\end{remark}

The next result gives an explicit formula for the probability of $\omega$ in terms of an associated Potts model.
\begin{corollary} \label{cor:marginal}
The marginal distribution on $(\omega,\sigma')$ is given by
\begin{align} \label{eq:od1}
\mathbf P (\omega,\sigma') \propto q^{k(\omega)}  a^{|\eta(\sigma')|}(1-b)^{|\omega \setminus \eta(\sigma')|} b^{|\mathsf E\setminus \omega| }\mathbf{1}_{\{\eta(\sigma')\subseteq \omega\}}, \quad \sigma' \in Q'^{\mathsf U}, \omega\subseteq \mathsf E.
\end{align}
Summing over all $\sigma'$, we get
\begin{align} \label{eq:od2}
\mathbf P (\omega) \propto q^{k(\omega)} \big(\tfrac{1-b}b \big)^{|\omega|}Z_{ (\mathsf V({ \omega}),\omega)^*,q'}(\tfrac{1-a-b}{a}), \qquad \omega\subseteq \mathsf E.
\end{align}
\end{corollary}

\begin{proof}
The first equality follows directly from \eqref{eq:condweight}, and the fact that there are exactly $q^{k(\omega)}$ configurations of $\sigma$ which are constant on the clusters of $\omega$.
We get the second equality from the fact that
\[
\sum_{\sigma': \ \eta(\sigma') \subseteq  \omega }  \big ( \tfrac{a}{1-b} \big)^{|\eta(\sigma')|} =Z_{ (\mathsf V({ \omega}),\omega)^*,q'}\big(\tfrac{1-a-b}{a}\big). \qedhere
\]
\end{proof}
Note that the Potts model whose partition function appears in~\eqref{eq:od2} is ferromagnetic if and only if $a+b\leq 1$.

\section{Relationship with other models} \label{sec:other}
For special values of the parameters $a,b$ and $q,q'$, we recover various well known models of statistical mechanics. 
\subsection{$\text{FK}$-random cluster model for $a+b=1$}
Recall that in this case $\omega'= \omega^{\dagger}$ almost surely, and hence \eqref{eq:weight1} simplifies to
\begin{align} \label{eq:FKraw}
\mathbf{P}(\omega,\omega',\sigma, \sigma') = \frac{1}{\mathcal Z} a^{|\omega|} (1-a)^{|\mathsf E\setminus \omega|}, \quad (\omega,\omega',\sigma,\sigma')  \in \Omega\Sigma_{1},
\end{align}
where $\Omega\Sigma_{1}= \Omega\Sigma_{\geq 1} \cap \Omega\Sigma_{\leq 1}$.

We first consider the case when $M$ is of genus zero.
We can readily recognize the underlying Fortuin-Kasteleyn random cluster model~\cite{FK} (see e.g.\ \cite{RC} for an exposition on this classical subject).

\begin{proposition} \label{prop:RC}
Assume that $M$ is of genus zero, and $a+b=1$. Let
\[ 
\qquad p =\frac{q'}{q'+a^{-1}-1}.
\] 
Let $k(\omega)$ be the number of clusters of $\omega$. 
Then the marginal distribution of $\mathbf{P}$ on $\omega$ is given by
\begin{align*}
\mathbf{P}(\omega) \propto  (qq')^{k(\omega)}p^{|\omega|}(1-p)^{|\mathsf E\setminus \omega|}, 
\end{align*}
which is the $\textnormal{FK}( qq')$ random cluster model measure on $\mathsf G$ with free boundary conditions. 
\end{proposition}
\begin{proof}
Consider $\omega \subseteq \mathsf E$. Using \eqref{eq:FKraw}, it is enough to count how many pairs of spin configurations $(\sigma,\sigma')\in \Sigma$ are compatible with $\omega$, meaning that 
$(\omega,\omega^{\dagger},\sigma,\sigma')\in \Omega\Sigma_{1}$. By the definition of $\Omega\Sigma_1$, this is the same as requiring that
$\sigma$ is constant on the connected components of $\omega$, and $\sigma'$ is constant on the connected components of $\omega^{\dagger}$. Using Euler's formula 
\[
k(\omega^{\dagger}) = k(\omega) + |\omega|-|V|+1
\] 
we conclude that the total number of compatible pairs is $(q {q'})^{k(\omega)} {q'}^{|\omega|} \times const$, where $const$ is independent of~$\omega$. Plugging this into \eqref{eq:FKraw} we get
\[
\mathbf{P}(\omega) \propto  (qq')^{k(\omega)}(aq')^{|\omega|}(1-x)^{|\mathsf E\setminus \omega|} \propto   (qq')^{k(\omega)}p^{|\omega|}(1-p)^{|\mathsf E\setminus \omega|},
\]
which concludes the proof.
\end{proof}

\begin{remark} 
\label{rem:fkcrit} For $q=q'$, the point $a=b=1/2$ corresponds to $p=q/(q+1)$ which is the critical point of the $\textnormal{FK}(q^2)$ random cluster model on the square lattice~\cite{BefDum}.
\end{remark}

\begin{corollary}
Assume that $M$ is of genus zero.
Let $\omega$ be distributed according to the $\textnormal{FK}( qq')$ random cluster model with parameter $p$ as above. For each cluster of $\omega$, choose a spin $\sigma \in  Q$,
and for each cluster of~$\omega^{\dagger}$, choose a spin $\sigma' \in  Q'$ uniformly and independently of one another. Then $(\sigma,\sigma')$ has the distribution of the spin model \eqref{eq:defspin} with $a$ as above and $a+b=1$.
\end{corollary}
\begin{proof}
This follows from the arguments in the proof above.
\end{proof}

\begin{remark}
\hspace{0cm}
\begin{itemize}
\item For $q'=1$ and $a$ arbitrary, there is no constraint on $\sigma$ and the coupling of $(\sigma,\omega)$ becomes the classical Edwards--Sokal coupling between the q-state Potts model and the $\textnormal{FK}(q)$ random cluster model~\cite{EdwSok}.
\item For $q=1$, $q'=2$ and $a+b=1$, $\eta(\sigma')$ is an \emph{even subgraph} of $\mathsf G$, meaning that the degree of every vertex in $(\mathsf V,\eta(\sigma'))$ is even, and $\omega$ is distributed like the $\textnormal{FK}(2)$ random cluster model. This coupling of $(\omega,\eta(\sigma'))$ is the same as in the work of
Grimmett and Janson~\cite{GriJan}.
\end{itemize}
\end{remark}

We now assume that $M$ is a torus. The necessary Euler's formula takes a slightly more complicated form in this case.
We follow the notation of~\cite[Section 4.3.2]{DCfermion}.
Define $\delta(\omega)\in\{0,1,2\}$ depending on the topology of $\omega$:
\begin{itemize}
\item if $\omega$ contains two non-contractible cycles of different homotopy, then ${\delta(\omega)=2}$;
\item if $\omega$ contains a non-contractible cycle and all such cycles are homotopic, then~$\delta(\omega)=1$;
\item if all connected components of $\omega$ are contractible, then~$\delta(\omega)=0$.
\end{itemize}
Note that $\delta(\omega)+\delta(\omega^{\dagger})=2$.
With this notation Euler's formula reads 
\begin{align} \label{eq:Euler1}
k(\omega^{\dagger})=k(\omega)+|\omega|-\delta(\omega)-|\mathsf V|+1.
\end{align}

Using the same arguments as above, we can prove the following result.

\begin{proposition} \label{prop:RCtorus}
Assume that $M$ is a torus. 
Then the marginal distribution of $\mathbf{P}$ on $\omega$ is given by
\begin{align*}
\mathbf{P}(\omega) \propto ( qq')^{k(\omega)}{q'}^{-\delta(\omega)}p^{|\omega|}(1-p)^{|\mathsf E\setminus \omega|},
\end{align*}
where $p$ is an in Proposition~\ref{prop:RC}.
\end{proposition}
In the case $q=q'$ this distribution is that of the \emph{balanced random cluster model} with parameter $q^2$ as defined in~\cite{DCfermion}.
This model, unlike the standard random cluster model defined on a torus, exhibits duality meaning that if $\omega$ is distributed according to a balanced random cluster model measure,
then so is~$\omega^{\dagger}$. This is clear from the above result, as $\omega'=\omega^{\dagger}$.

\subsection{The staggered six-vertex model for $q=q'=2$} \label{sec:6V}
Let $\mathsf G^{\times}$ be the \emph{medial graph} of~$\mathsf G$ where a vertex is placed at the intersection of each primal edge and its dual, and where two vertices 
are adjacent if the corresponding pair of primal or dual edges share an endpoint. Note that the medial graph is $4$-regular, and its faces are in a natural correspondence with $\mathsf V\cup \mathsf U$ -- the vertices and faces of $\mathsf G$.
Moreover, the dual graph $(\mathsf G^{\times})^*$ is bipartite since the faces of $\mathsf G^{\times}$ corresponding to $\mathsf V$ can only be adjacent to the faces corresponding to~$\mathsf U$ and vice versa.
Also note that the same medial graph is obtained if we start with $\mathsf G^*$ instead of~$\mathsf G$.

Let $\mathcal{O}$ be the set of assignments to each edge of $\mathsf G^{\times}$ an orientation in such a way that there are exactly two incoming and two outgoing edges at each vertex of $\mathsf G^{\times}$.
We say that an element of $ \mathcal{O}$ is an \emph{arrow configuration}.
The \emph{(zero field) staggered} \emph{six-vertex model} with parameters $a,b>0$ (here we assume that the third parameter is $c=1$) is a probability measure on arrow configurations proportional to
\begin{align} \label{eq:def6V}
 a^{N_1} b^{N_2},
\end{align}
where $N_1$ and $N_2$ are the numbers of vertices of $\mathsf G^{\times}$ with the local arrow arrangements of type 1 and 2 respectively~\cite{Pau,Lieb,Lieb1}. 
The three types of local arrangements (each one has two subtypes) are pictured in Fig.~\ref{fig:6v}. 
\begin{figure}
\begin{center}
 \includegraphics[scale=0.9]{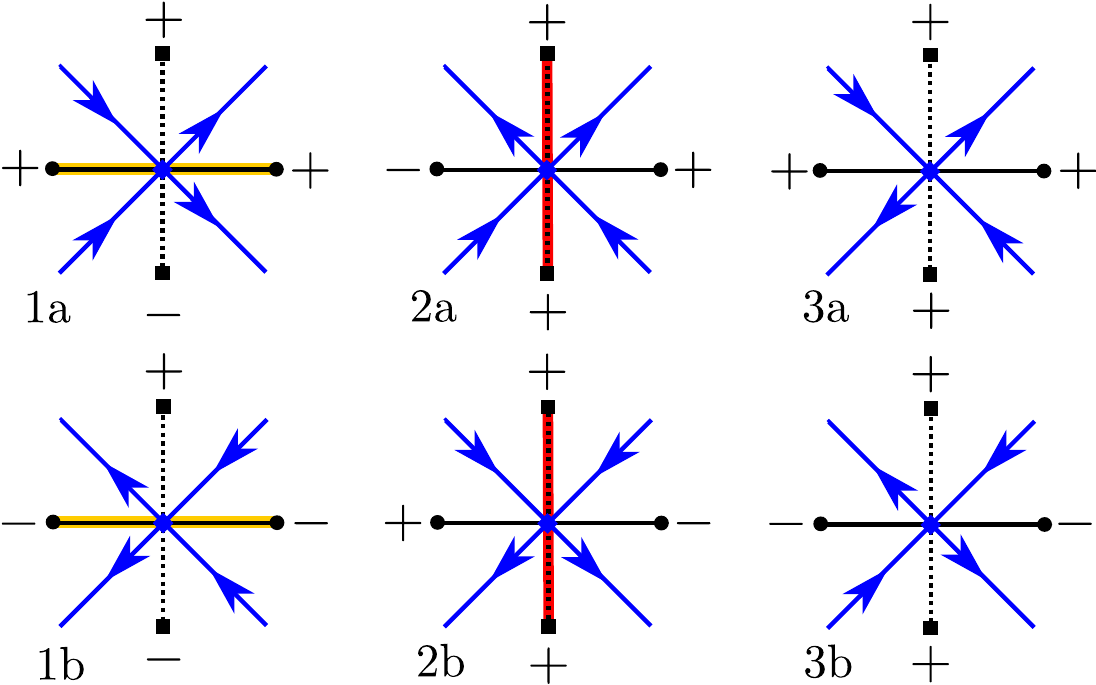}
\caption{A primal edge (solid), its dual edge (dashed), and four corresponding medial edges (blue). The figure shows the three 
types (up to arrow reversal) of local arrow arrangements in the six-vertex model on the medial graph. The sets of yellow primal and red dual edges $\eta'$ and $\eta$ are given by Rys' mapping.
The signs are the values of $\tilde \sigma$ and $\tilde \sigma'$ from \eqref{eq:spindef}. We have $\eta'=\eta(\tilde \sigma')$ and $\eta=\eta(\tilde \sigma)$ }
 \label{fig:6v}
\end{center}
\end{figure}

An observable of interest in the six-vertex model is its \emph{height function} $\tilde H$ defined on the faces of $\mathsf G^{\times}$, or equivalently on $\mathsf V \cup \mathsf U$.
It is given by first fixing its value at a chosen face $u_0$ of $\mathsf G^{\times}$. Then for any other face $u$,
one draws a directed path $\gamma$ in the dual of $\mathsf G^{\times}$ (which is the \emph{quad graph} of $\mathsf G$ and $\mathsf G^*$) connecting $u_0$ to $u$, and one defines
$\tilde H^{\gamma}_{\leftarrow}(u)$ and $\tilde H^{\gamma}_{\rightarrow}(u)$ to be the numbers of arrows in the 
underlying six-vertex configuration that cross $\gamma$ from right to left, and from left to right respectively.
The height at $u$ is then given by
\begin{align} \label{eq:hf}
\tilde H(u) =\tilde H^{\gamma}_{\leftarrow}(u) - \tilde H^{\gamma}_{\rightarrow}(u).
\end{align}
In the case of genus zero, the right-hand side is independent of $\gamma$ since the surface is simply connected and six-vertex configurations form conservative flows.
In higher genus the value of $H(u)$ in general depends on the homotopy class of the path $\gamma$, and can be thought of as a function defined on the faces of the universal cover of~$\mathsf G^{\times}$.
Note that the height function has a fixed parity on $\mathsf G$ and $\mathsf G^*$.

We define $\mathcal{O}^0$ to be the set of arrow configurations for which the increment of the height function along any closed path in the dual of $\mathsf G^{\times}$
is equal to zero $\textnormal{mod} \ 4$. Note that $\mathcal{O}^0=\mathcal O$ if $M$ is of genus zero.

Rys in~\cite{Rys} introduced a correspondence between~$\mathcal{O}^0$ and the set of pairs $(\eta,\eta')$ of certain primal and dual subgraphs that do not intersect (see Fig.~\ref{fig:6v}).
We note that this representation appeared also in the work of Nienhuis \cite{Nie}, and Boutillier and de Tili\`{e}re \cite{BoudeT} where it was used to represent the double Ising model
as a free-fermion six-vertex model. 
\begin{lemma}[Rys' mapping]\label{lem:Rys}
Let $q=q'=2$.
Then for each arrow configuration in $\mathcal{O}^0$, there exist exactly two spin configurations $(\tilde \sigma,\tilde \sigma')\in \Sigma$ that differ by a global sign change, such that $\eta=\eta(\tilde\sigma)$, and $\eta'=\eta(\tilde \sigma')$,
where $\eta'$ and $\eta$ are the sets of primal and dual edges respectively defined from the arrow configuration as in Fig.~\ref{fig:6v}.
\end{lemma}
This correspondence is a consequence of the following construction.
We chose to fix the height function $\tilde H$ to be $\pm 1$ with equal probability at a fixed face $u_0\in \mathsf U$. 
We now define spins $\tilde \sigma$ on $\mathsf V$ and $\tilde \sigma'$ on $\mathsf U$ by
\begin{align} \label{eq:spindef}
\tilde{\sigma}(v) = i^{\tilde H(v)},  \qquad \textnormal{and} \qquad \tilde{\sigma}'(u) = i^{\tilde H(u)+1}.
\end{align}
This definition depends only on the values of $\tilde H$ $\textnormal{mod } 4$, and since we consider only configurations from $\mathcal O^0$, $\tilde \sigma$ is a well defined function on the faces of $\mathsf G^{\times}$.
We first claim that $(\tilde \sigma,\tilde \sigma' )\in \Sigma$, which can be checked by inspection in Fig.~\ref{fig:6v}, where the spin $\tilde \sigma'$ at the top was fixed to be $+1$. On the other hand,
knowing the spins $(\tilde \sigma,\tilde \sigma' )\in \Sigma$ one can locally and consistently recover an arrow-configuration constructed as in Fig.~\ref{fig:6v}.

Moreover, by our symmetric choice of $\tilde H(u_0)$ the distribution of 
$(\tilde \sigma,\tilde \sigma')$ is invariant under a global sign reversal, and since the weights from \eqref{eq:def6V} agree with those from \eqref{eq:defspin},
we readily get the following correspondence.
\begin{corollary} \label{cor:6vrep}
The law of $(\tilde \sigma,\tilde \sigma')$ induced by the map \eqref{eq:spindef} from the law of the staggered six-vertex model on $\mathsf G^{\times}$ conditioned on $\mathcal O^0$ has the same distribution as $(\sigma,\sigma')$ under $\mathbf P$. Moreover, the respective height functions $\tilde H$ and $H$ have the same distribution up to a global additive constant.
\end{corollary}

\begin{remark}
As mentioned before, our results from previous sections generalize to the case of nonhomogeneous weights $a_e$ and $b_e$ that depend on the particular edge.
To obtain the classical (non-staggered) six-vertex model on the square lattice instead of the staggered one, 
one has to assign weights $a$ to the horizontal edges and $b$ to the vertical edges $\mathsf G$.\end{remark}

\subsection{Random currents for $q'=2$ at the the free fermion point $a^2+b^2=1$}
Here we again assume that $M$ is of genus zero.
A \emph{current} is simply a function $\n : \mathsf E \to \{0,1,\ldots,\}$. For a current~$\n$, define $\omega(\n)$ and $\eta(\n)$ to be the set of edges with non-zero 
and odd values of $\n$ respectively. In particular, $\eta\subseteq \omega$. We will say that the pair $(\eta(\n),\omega(\n))$ is the \emph{trace} of the current $\n$. 
We will often identify a current with its trace as the trace contains all the relevant information for the probability measures that we will consider.

Let $\mathcal{E}$ be the collection of sets of edges $\eta$ such that each vertex in $\mathsf V$ has even degree 
in the graph $(\mathsf V,\eta)$. Such $\eta$ are usually called \emph{even subgraphs} of $\mathsf G$. By $\Gamma$ we will denote the set of all possible 
traces $(\omega,\eta)$ of currents such that $\eta\in \mathcal{E}$.
The currents in $\Gamma$ are commonly called \emph{sourceless}. Note that since $M$ is either a sphere or the plane, spin configurations 
$\sigma': U \to \{-1,+1\}$ are in a 2-to-1 correspondence with even subgraphs given by the map $\sigma'\mapsto \eta(\sigma')$. This is not true for higher genera since
then not all even subgraphs (e.g.\ a non-contractible cycle on a torus) can be realized as the set of interfaces of a spin configuration.
In what follows we will write $\eta=\eta(\sigma')$.

\subsubsection*{Single random current for $q=1$}
The single random current measure is induced from the power series expansion of the Ising model partition function (see e.g.\ \cite{LisT}). 
We have the following observation:
\begin{corollary} 
Assume that $M$ is of genus zero. Let $a^2+b^2=1$, $q'=2$ and $q=1$. Then
\[
\mathbf{P}(\eta,\omega)\propto a^{|\eta|} (1-b)^{|\omega\setminus \eta|}b^{|\mathsf E\setminus \omega|}, \qquad (\eta,\omega) \in \Gamma,
\]
which is the law of the sourceless single random current with $a = \tanh J$.
\end{corollary}
\begin{proof}
The formula for the probability is a direct consequence of Corollary~\eqref{cor:marginal}. The identification with the random current measure follows for 
instance from Lemma~3.1 of \cite{LisT}, by setting $a=\tanh J$ and $b=(\cosh J)^{-1}$.
\end{proof}

\subsubsection*{Double random current for $q=2$}
The double random current can be defined as the sum of two i.i.d.\ sourceless single random currents~\cite{GHS}.
It turns out that going from a single current to a double current amounts to changing $q$ from $1$ to~$2$, and $J$ to $2J$:
\begin{corollary} \label{cor:doublecurr}
Assume that $M$ is of genus zero. Let $ x\in (0,1]$ be given by $a=2 x /(1+ x^2)$. Moreover, let $a^2+b^2=1$ and $q'=q=2$. Then 
\[
\mathbf{P}(\eta,\omega) \propto 2^{k(\omega)+|\omega|}  x^{|\eta|}  (x^2)^{|\omega\setminus \eta|}(1- x^2)^{|E\setminus \omega|}, \qquad  (\eta,\omega) \in \Gamma,
\]
which is the law of the sourceless double random current with $ x = \tanh J$, or equivalently $a =\tanh 2J$.
\end{corollary}
\begin{proof}
Again, the formula for the probability follows directly from Corollary~\eqref{cor:marginal}. The identification with the double random current measure follows 
from Theorem~3.2 of \cite{LisT}, by setting $x=\tanh J$ and using the fact that since $a^2+b^2=1$, we have $b = {(1- x^2)}/{(1+ x^2)}$. 
\end{proof}

\begin{remark}
Recall that the XOR-Ising model is the pointwise product of two i.i.d. Ising model spin configurations~\cite{Wil,Dub,BoudeT}.
It is known that in the free fermion case $a^2+b^2=1$, the distribution of $\sigma$ and $\sigma'$ is that of the XOR-Ising model and its dual XOR-Ising model~\cite{BoudeT}.
It follows from Corollary~\ref{cor:doublecurr} and Proposition~\ref{prop:coupling} that an independent assignment of a $\pm 1$ spin to the clusters of a double random current
yields the XOR-Ising model configuration. This result was, to our knowledge, first observed during a discussion of Roland Bauerschmidt, Hugo Duminil-Copin, Aran Raoufi and the 
author at IHES in 2017, and was the main inspiration for the 
considerations in this article.
\end{remark}

\subsection{Loop $O(n)$ model for $q=n$, $q'=2$ and $b=1$}
Assume that $M$ is of genus zero, and let $\mathsf G$ be a $3$-regular graph e.g.\ a piece of the hexagonal lattice. For each spin configuration, $\sigma': \mathsf U\to \{-1,1\}$, $\eta(\sigma')$ is a collection of 
disjoint loops on $\mathsf G$, and each such collection of loops corresponds to exactly two spin configurations. This implies that for $b=1$, the marginal distribution on $\eta=\eta(\sigma')$ is given by
\[
\mathbf{P}(\eta) \propto n^{k(\eta)} a^{|\eta|} \propto n^{\# \text{ loops in }\eta }  x ^{|\eta|}, \qquad \eta \in \mathcal{E},
\]
where $ x = a/n$, and where $\mathcal E$ is the set of all collections of disjoint loops on~$\mathsf G$. 
In the second identity, we used the fact that each vertex of $\mathsf G$ is either isolated in $(\mathsf V,\eta)$ or belongs to exactly two edges of $\eta$
This is the law of the \emph{loop $O(q)$ model} with parameter $ x$ (see e.g.\ \cite{PelSpi} for an exposition on the subject). 

Note that the spin model $\sigma'$ for $q=1$ is the standard Ising model. We note that for real valued $n\geq 1$, this model has been studied in \cite{DCetal}, 
where its FKG property (valid for $ x\leq \frac1{\sqrt{n}}$) was used to prove 
existence of macroscopic loops at the Nienhuis' critical point.

\section{$\sigma$ and $\sigma'$ as classical spin models} \label{sec:unconstrained}
In this section we assume that $M$ is of genus zero.
It turns out that the spin model $\sigma$ can be represented 
via a classical unconstrained spin system, which is a special case of the model of Domany and Riedel~\cite{DomRie}, where two Potts models are coupled via a general four-spin interaction. This property is a consequence of duality in the Potts model,
and is a generalization to arbitrary $q$ and $q'$ of the six-vertex model representation of the Ashkin-Teller model~\cite{Nie} (see also~\cite{Weg,BoudeT}).
We note that similar ideas applied to the model from~\cite{DomRie} are present in~\cite{Perk}.

We consider a spin model on configurations $(\s,\s')\in Q^{\mathsf V}\times{Q'}^{\mathsf V}$ given by the Gibbs-Boltzmann distribution
\begin{align} \label{eq:sspin}
\tilde \mu(\s,\s') \propto \exp \Big( \sum_{\{v_1,v_2\} \in \mathsf E} \delta_{\s(v_1),\s(v_2)} \big (\alpha +\beta \delta_{\s'(v_1),\s'(v_2)}\big)\Big ),
\end{align}
where 
\begin{align}
\alpha =\ln\big(\tfrac{1-a}b\big) \qquad \text{and} \qquad \beta= \ln \big(1+\tfrac{q'a}{1-a} \big), \label{eq:parameters}
\end{align}
and where $\delta_{x,y}=1$ if $x=y$ and $\delta_{x,y}=0$ otherwise.
Note that $\beta$ is always positive, and $\alpha$ is positive if and only if $a+b\leq 1$. 
These parameters form a two-dimensional subspace of the parameters of the general model from~\cite{DomRie}.

\begin{remark}
This spin model can be defined on any finite graph, not necessarily embedded in a surface.
\end{remark}

The following values of parameters are of special interest:
\begin{itemize}
\item The case $\alpha=0$, or equivalently $a+b=1$, corresponds to a ferromagnetic $qq'$-state Potts model $\tilde \s$ with $J=\beta>0$. Then, $\s= \tilde \s \ (\textnormal{mod}\ q)$ in distribution.
\item For $q'=2$, the case $\alpha=-\beta$, or equivalently $b-a=1$, on a bipartite graph corresponds to an antiferromagnetic $2q$-state Potts model $\tilde \s$ with $J=-\beta$.
Indeed, if we flip the value of the $\s'$ spin on all black vertices and call the resulting spin configuration $\s''$, we have $1-\delta_{\s'(v_1),\s'(v_2)} =\delta_{\s''(v_1),\s''(v_2)}$.
Again, $\s= \tilde \s \ (\textnormal{mod}\ q)$ in distribution.
\item The case $2\alpha=-\beta$, or $a^2+b^2=1$, results in symmetric energy levels $-\tfrac{\beta}2,0,\tfrac{\beta}2$ per bond. For $q=q'=2$, it corresponds to the free fermion point  in the six-vertex model. 
In this case one can write the Hamiltonian as a sum of two independent Hamiltonians, and as a result represent the system as two statistically independent copies of the 2-state Potts model (the Ising model).
For $q'=2$ and $q>2$, the model is equivalent to the $q$-component version of the cubic model of Kim, Levy and Uffer~\cite{cubic}.
\end{itemize}

Note that by definition~\eqref{eq:sspin}, conditioned on the spins $\s$, the spins $\s'$ do not interact along edges with a different value of $\s$ assigned to both endpoints. This is exactly the dual picture (zero coupling constant) of the hard-core constraint~\eqref{eq:divfree} (infinite coupling constant) for the spin model $\sigma$.
Hence, the following identification of $\sigma$ and $\s$ which is the main result of this section should not be surprising.
\begin{theorem} \label{thm:unconstrained}
Assume that $M$ is of genus zero.
Then the distributions of $\sigma$ under $\mathbf P$, and of $\s$ under $\tilde \mu$ are the same.
\end{theorem}

Before the proof, we need to recall the high-temperature expansion of the Potts model partition function and its planar duality. 
Let $G=(V,E)$ be a finite and not necessarily connected graph embedded in $M$, and let $G^*=(U,E^*)$ be its dual graph. 
We will often drop the parameter $q$ from the notation and write $Z_G$ for $Z_{G,q}$.
Recall that $x+1=e^J$. We have
\begin{align} \nonumber
(x+1)^{|E|}Z_G(x) &= \sum_{\s \in Q^V }\exp \Big(J \sum_{\{v_1,v_2\}\in E} \delta_{\s(v_1),\s(v_2)} \Big)  \\ \nonumber
&= \sum_{\s \in Q^V } \prod_{\{ v_1,v_2\}\in E} (1+\delta_{\s(v_1),\s(v_2)} x) \\ \nonumber
&= \sum_{\s \in Q^V } \sum_{\xi \subseteq E} x^{|\xi|}\mathbf{1}\{ \s \ \textnormal{constant on clusters of } \xi \} \\ \label{eq:highT}
&= \sum_{\xi \subseteq E} x^{|\xi|}q^{k(\xi)},
\end{align}
where $k(\xi)$ is the number of {clusters} of $\xi$ in $G$, i.e., connected components of the graph $(V,\xi)$ including isolated vertices.
Again recall that Euler's formula for planar graphs reads
\begin{align} \label{eq:Euler}
k(\xi) = f(\xi) -|\xi|+|V|-1,
\end{align}
where $f(\xi)$ is the number of faces of $\xi$. Note that if we denote $\xi^{\dagger}=E^*\setminus \xi^*$, then $f(\xi)=k(\xi^{\dagger})$ is the number of clusters of $\xi^{\dagger}$ in $G^*$. 
Hence by \eqref{eq:highT} we can write
\begin{align} \nonumber
Z_G(x)  &= (x+1)^{-|E|}q^{|V|-1}\sum_{\xi \subseteq E} (\tfrac{x}q)^{|\xi|}q^{k(\xi^{\dagger})} \\ \nonumber
&= q^{|V|-|E|-1}\big( \tfrac{x}{x+1}\big)^{|E|}\sum_{\xi^{\dagger} \subseteq E^*} (x^*)^{|\xi^{\dagger}|}q^{k(\xi^{\dagger})} \\
&=q^{|V|-|E|-1}\big( \tfrac{x}{x+1}\big)^{|E|} (x^*+1)^{|E^*|}Z_{G^*}(x^*) \nonumber \\ 
&=q^{|V|-|E|-1}\big( \tfrac{x+q}{x+1}\big)^{|E|} Z_{G^*}(x^*)  \label{eq:Pottsduality},
\end{align}
where $x^* = q/x$, and where in the third equality we again used \eqref{eq:highT}.

Based on Corollary~\ref{cor:marginal}, we can now prove an intermediate result which gives a formula for the probability of 
$\omega$ in terms of the (primal) Potts model partition function $Z_{(\mathsf V,\omega),q'}$.
\begin{proposition} \label{prop:negform}Assume that $M$ is of genus zero. Then, the marginal distribution of $\mathbf P$ on $\omega$ is given by 
\begin{align}  \label{eq:negform}
\mathbf P (\omega) \propto q^{k(\omega)} \big(\tfrac{1-(1-q')a-b}b\big)^{|\omega|}Z_{(\mathsf V,\omega),q'}(\tfrac{q'a}{1-a-b}), \qquad \omega \subseteq \mathsf E.
\end{align}
\end{proposition}
\begin{proof} 
Let $x^*=\tfrac{1-a-b}{a}$  and $x=\tfrac{q'}{x^*}=\tfrac{q'a}{1-a-b}$, and note that 
\[
Z_{ (\mathsf V(\omega),\omega)}=q^{|V( \omega)|-|\mathsf V|} Z_{(\mathsf V, \omega)}.
\]
Applying duality we get
\begin{align} \label{eq:duality2}
Z_{  (\mathsf V(\omega),\omega)}(x) = q'^{|V( \omega)| - | \omega|-1} \big( \tfrac{x+q'}{x+1}\big)^{| \omega|} Z_{  (\mathsf V(\omega),\omega)^*}(x^*).
\end{align}
Therefore, by Corollary~\ref{cor:marginal} we have
\begin{align*}
q^{-k(\omega)} \big(\tfrac{b}{1-b}\big)^{|\omega|}  \mathbf P (\omega) & \propto Z_{  (\mathsf V(\omega),\omega)^*}(x^*) \\
& =   {q'}^{-|V( \omega)| + | \omega|+1}\big( \tfrac{x+1}{x+q'}\big)^{| \omega|} Z_{  (\mathsf V(\omega),\omega)}(x) \\
&={q'}^{-|V( \omega)| + | \omega|+1+| V( \omega)|-|\mathsf V|}  \big( \tfrac{x+1}{x+q'}\big)^{| \omega|}  Z_{(\mathsf V,\omega)}(x) \\
&= \big(\tfrac{1-(1-q')a-b}{1-b}\big)^{| \omega|}   Z_{(\mathsf V,\omega)}(x),
\end{align*}
where in the last equality we used the identity $q' \tfrac{x+1}{x+q'}=\tfrac{1-(1-q')a-b}{1-b}$.

\end{proof}
Note that for $a+b>1$ and $b\leq 1$, we have that $x< -1$, and hence the Potts model corresponding to the partition function $Z_{(\mathsf V, \omega),q'}(x)$ has a complex-valued coupling constant $J$.
To complete the proof of Theorem~\ref{thm:unconstrained} we need the following elementary lemma which effectively turns this coupling constant into a real-valued one.
\begin{lemma} \label{lem:resum}
Let $G=(V,E)$ be a finite graph. Then for all $t\neq-1$, we have
\begin{align*}
\sum_{\xi \subseteq E }  t^{|\xi|}(x+1)^{|\xi|}  Z_{(V,\xi)} (x)  =(1+t(x+1))^{|E|}  Z_{G}\big(\tfrac{tx}{1+t}\big).
\end{align*}
\end{lemma}
\begin{proof}
We have
\begin{align*}
\sum_{\xi \subseteq E } t^{|\xi|}(x+1)^{|\xi|}  Z_{(V,\xi)} (x)& = \sum_{\xi \subseteq E }t^{|\xi|} \sum_{\xi'\subseteq \xi} x^{|\xi'|}q^{k(\xi')} \\
&= \sum_{\xi' \subseteq E }\sum_{\xi \supseteq \xi'}  t^{|\xi|}  x^{|\xi'|}q^{k(\xi')} \\
&= \sum_{\xi' \subseteq E } (1+t)^{|E|-|\xi'|}  (tx)^{|\xi'|}q^{k(\xi')} \\
& = (1+t)^{|E|} \sum_{\xi' \subseteq E } \big(\tfrac{tx}{1+t})^{|\xi'|}q^{k(\xi')}\\
&=(1+t(x+1))^{|E|} Z_{G}\big(\tfrac{t x}{1+t}\big),
\end{align*}
where in the first and last equality we applied the high-temperature expansion~\eqref{eq:highT}.
\end{proof}

We can finally identify the distribution of $\sigma$ as that of $\s$.
\begin{proof}[Proof of Theorem~\ref{thm:unconstrained}]

For $i \in Q$ and $\s\in Q^{\mathsf V}$, we define 
\[
E_i(\s) = \{ \{v_1,v_2\}\in \mathsf E: \ \s(v_1)=\s(v_2)=i \} \quad \text{and} \quad E(\s) = \bigcup_{i\in Q} E_i(\s).
\]
For $\xi\subseteq \mathsf V$, we denote  $\tilde{Z}_{\xi}=Z_{( \mathsf V(\xi),\xi),q'}$.
On one hand, by the definition of the spin model~\eqref{eq:sspin}, conditioned on $\s$, the $\s'$ spins do not interact whenever the corresponding $\s$ spins are different, and hence we have
\begin{align}
\tilde \mu(\s) & \propto (q')^{|\mathsf V|- | V(E(\s))|}  e^{\alpha |E(\s)|}\prod_{i\in Q} e^{\beta |E_i(\s)|}\tilde Z_{E_i(\s)} (e^{\beta}-1) \nonumber \\
&=  (q')^{|\mathsf V|- | V(E(\s))|} e^{(\alpha+\beta) |E(\s)|}\prod_{i\in Q}  \tilde Z_{E_i(\s)} \big (\tfrac{q'a}{1-a} \big) \nonumber \\
&=  (q')^{|\mathsf V|- | V(E(\s))|}  \big(\tfrac{1-(1-q')a}{b}\big)^{ |E(\s)|}\prod_{i\in Q}  \tilde Z_{E_i(\s)} \big (\tfrac{q'a}{1-a} \big). \label{eq:prob1}
\end{align}
On the other hand, by property \eqref{eq:step1} from Proposition~\ref{prop:coupling}, and Proposition~\ref{prop:negform}, we have
\begin{align}
\mathbf P (\sigma)& \propto \sum_{\omega \textnormal{ consistent with } \sigma}\big(\tfrac{1-(1-q')a-b}b\big)^{|\omega|}Z_{(\mathsf V,\omega)}(\tfrac{q'a}{1-a-b}) \nonumber \\
&= \sum_{\omega \textnormal{ consistent with } \sigma} \big(\tfrac{1-(1-q')a-b}b\big)^{|\omega|}   (q')^{|\mathsf V|- | V(\omega)|}  \prod_{i\in Q} \tilde Z_{E_i(\sigma) \cap \omega} (\tfrac{q'a}{1-a-b}) \nonumber \\
&= (q')^{|\mathsf V|- | V(E(\sigma))|}   \prod_{i\in Q} \sum_{\xi \subseteq E_i(\sigma)}  \big(\tfrac{1-(1-q')a-b}b\big)^{|\xi|}  (q')^{| V(E_i(\sigma))|- | V(\xi)|} \tilde Z_{\xi} (\tfrac{q'a}{1-a-b}) \nonumber \\
&=(q')^{|\mathsf V|- | V(E(\sigma))|}  \big(\tfrac{1-(1-q')a}{b}\big)^{ |E(\s)|} \prod_{i\in Q}  \tilde Z_{E_i(\s)} \big (\tfrac{q'a}{1-a} \big), \label{eq:prob2}
\end{align}
where `$\omega$ {consistent with} $\sigma$' means that $\sigma$ is constant on the clusters of $\omega$, and where in the last equality we used Lemma~\ref{lem:resum} with $t=(1-a-b)/b$ and
$x={q'a}/({1-a-b})$. The formulas \eqref{eq:prob1} and \eqref{eq:prob2} are identical and we finish the proof.
\end{proof}

 \section{Positive association} \label{sec:posass}
We discuss here the positive association of the percolation and spin models. To this end, we first briefly recall the basic notions of this theory.
For a set~$S$, we identify $\mathcal P(S):=\{0,1\}^S$ with the set of subsets of $S$, and for $\xi_1,\xi_2 \in \mathcal P(S)$, we write $\xi_1\leq \xi_2$ if $\xi_1\subseteq \xi_2$.
We consider $\mathcal P(S)$ as a probability space with a probability measure $\nu$.
A random variable $X:\mathcal P(S)\to \mathbb R$ is called \emph{increasing} (resp.\ \emph{decreasing}), if $X(\xi_1)\leq X(\xi_2)$ (resp.\ $X(\xi_1)\geq X(\xi_2)$) whenever {$\xi_1\leq \xi_2$}.
Similarly, an event $A \subseteq \mathcal P(S)$ is called increasing (resp.\ decreasing) if its indicator function is increasing (resp.\ decreasing).

We call $\nu$ \emph{strictly positive} if $\nu(\xi)>0$ for all $\xi\in \mathcal P(S)$.
We say that $\nu$ satisfies the \emph{FKG inequality}, or is \emph{positively associated}, if
\[
\nu(XY) \geq \nu(X)\nu(Y)
\]
for all increasing random variables $X,Y$.
Moreover, we say that $\nu$ satisfies the \emph{FKG lattice condition} if for all $\xi_1,\xi_2 \in \mathcal P (S)$,
\begin{align} \label{eq:FKGlattice}
\nu(\xi_1\wedge \xi_2)\nu(\xi_1\vee \xi_2) \geq \nu(\xi_1) \nu(\xi_2),
\end{align}
where $\xi_1\wedge \xi_2$, resp.\ $\xi_1\vee \xi_2$, denote the pointwise minimum, resp.\ maximum, of $\xi_1$ and $\xi_2$ (or the intersection, resp.\ union, in the set interpretation of $\xi_1$ and~$\xi_2$).
If $\nu$ is strictly positive, then the FKG lattice condition implies positive association of $\nu$~\cite{FKG, Holley}, and moreover
it implies a stronger property, called \emph{strong positive association}, which e.g.\ yields stochastic monotonicity of the measure with respect to the imposed boundary conditions, see~\cite{RC}.

For $\xi \in \mathcal P(S)$ and $x\in S$, let $\xi^x,\xi_x\in \mathcal P(S)$ be the two configurations equal to $\xi$ on $S\setminus \{ x\}$, and satisfying $\xi^x(x)=1$ and $\xi_x(x)=0$.
It is known that for strictly positive measures, the FKG lattice condition holds if and only if
\begin{align} \label{eq:Grimm}
\nu (\xi^{x,y}) \nu(\xi_{x,y}) \geq \nu (\xi^{x}_{y})\nu(\xi^{y}_x)
\end{align}
for all $\xi\in \mathcal P (S)$ and $x,y \in S$ (see ~\cite{RC}).

\subsection{Positive association of the percolation model for $a+b\leq 1$} \label{sec:omegaFKG}
In this section we will use the above vocabulary with $S=\mathsf E$, and hence $\mathcal P(S)=\Omega$.
\begin{remark}
The following main result of this section, in the special case $q=q'=2$ and $a=b\leq \tfrac 12$, was independently proved in~\cite{SpiRay}. 
The partial order considered in~\cite{SpiRay} is actually different, and the monotonicity obtained there is stronger than ours in the sense that any increasing function according to our definition 
is also increasing according to the order of~\cite{SpiRay}, but not the other way around. We refer the interested reader to~\cite{SpiRay} for details.
\end{remark}

\begin{proposition} \label{prop:fkgomega}
The marginal distribution of $\mathbf P$ on $\omega$ satisfies the FKG lattice condition for $a+b\leq 1$.
\end{proposition}
\begin{proof}
By~\eqref{eq:Grimm}, it is enough to show that
\begin{align*} 
\mathbf P (\omega^{e,f})\mathbf P(\omega_{e,f}) \geq \mathbf P(\omega^{e}_{f})\mathbf P(\omega^{f}_e)
\end{align*}
for all $\omega\in \Omega$ and $e,f\in \mathsf E$. We will use formula~\eqref{eq:od2} from Corollary~\ref{cor:marginal}. Since
\[
k(\omega^{e,f})+k(\omega_{e,f}) \geq  k(\omega^{e}_{f}) + k(\omega^{f}_e),
\]
it is enough to prove that 
\begin{align}\label{eq:zzzz}
{Z_{G^*(\omega_{e,f})}}{Z_{G^*(\omega^{e,f})}} \geq {Z_{G^*(\omega^{e}_f)}}{Z_{G^*(\omega^{f}_e)}},
\end{align}
where $G^*(\omega) := (\mathsf V(\omega), \omega)^*$.
Recall that a bridge of $\omega$ is an edge in $\omega$ which does not belong to any cycle of $\omega$.
Note that if either $e$ or $f$ is a bridge of $\omega^{e,f}$, then \eqref{eq:zzzz} actually becomes an equality. Hence, we can assume that neither $e$ nor $f$ is a bridge of $\omega^{e,f}$.
In this case, let $\mu$ be the law of the ferromagnetic $q'$-state Potts model on $ G^*({\omega^{e,f}})$ with parameter $x=(1-a-b)/a$, and let $u_1$, $u_2$ (resp.\ $u_3$, $u_4$) be the faces of $ G({\omega^{e,f}})$ incident on $e$ (resp.\ $f$).
Then, after dividing both sides by $(Z_{G^*(\omega^{e,f})})^2$, \eqref{eq:zzzz} is equivalent to 
\[
\mu(\s(u_1)=\s(u_2), \s(u_3)=\s(u_4)) \geq \mu(\s(u_1)=\s(u_2)) \mu(\s(u_3)=\s(u_4)),
\]
which follows from the first inequality in the next lemma.
\end{proof}

\begin{lemma} \label{lem:clumping}
Let $\mu$ be the law of a ferromagnetic $q$-state Potts model on a finite graph $G=(V,E)$, and let $A,B\subseteq V$. Then
\[
\mu(\s \textnormal{ constant on } A \textnormal{ and on } B) \geq \mu(\s \textnormal{ constant on } A) \mu(\s \textnormal{ constant on }  B),
\]
and 
\[
\mu(\s \textnormal{ constant on } A\cup B) \geq \tfrac 1q \mu(\s \textnormal{ constant on } A\textnormal{ and on } B).
\]
\end{lemma}
\begin{proof}
Consider the FK-random cluster model associated to the Potts model in the Edwards--Sokal coupling (see e.g.\ \cite{RC} for details). 
Let $k(A)$ be the number of clusters of the configuration intersecting $A$, and let $ \{ A \con{} B\}$ be the event that $A$ and $B$ belong to the same cluster.
From the Edwards--Sokal coupling it follows that
\begin{align} \label{eq:ES1}
\mu(\s \textnormal{ constant on } A ) &= \mathbb E[q^{1-k(A)}], 
\end{align}
and 
\begin{align}\label{eq:ES2}
\mu(\s \textnormal{ constant on }  A  \text{ and on } B) =  \\ \nonumber
\mathbb E\big[ q^{1-k(A \cup B) } &  \mathbf 1\{ A \con{} B\} +q^{2-k(A)-k( B) } \mathbf 1\{ A \ncon{} B\}\big], 
\end{align}
where $\mathbb E$ is the expectation in the random cluster model. 
Using the fact that $k(A\cup B) \mathbf 1\{ A \con{} B\} \leq k(A)+k(B)-1$, we therefore get that
\begin{align}
\mu(\s \textnormal{ constant on } A \text{ and on } B) \geq \mathbb E\big [q^{2-k(A)-k( B) } \big],
\end{align}
and the first desired inequality follows from \eqref{eq:ES1} and the positive association of the FK-random cluster model (see e.g.\ \cite{RC}) applied to the increasing function $q^{-k(A)}$.
On the other hand, the second desired inequality follows from \eqref{eq:ES1}, \eqref{eq:ES2} and the fact that $k(A\cup B) \leq k(A)+k(B)$.
\end{proof}

\begin{remark}
One can check that the FKG lattice condition for $\omega$ does not hold when $a+b>1$. The reason is that the associated Potts model is in this case antiferromagnetic 
and the inequalities above are not valid anymore.
\end{remark}
\subsection{Positive association of the spin model $\sigma$ for $q=2$} \label{sec:spinFKG}
We now fix $q=2$, identify $Q$ with $\{0,1\}$, and set $S=\mathsf V$.

The following is the main result of this section. In the case $q'=2$ it was first proved by Glazman and Peled in~\cite{GlaPel}. 

\begin{proposition} 
Assume that $M$ is of genus zero and $q=2$. Then, the marginal distribution of $\tilde \mu$ on $\s$, and hence, by Theorem~\ref{thm:unconstrained}, also the marginal of $\mathbf P$ on~$\sigma$, satisfies the FKG lattice condition whenever $b\leq1$.
\end{proposition}

\begin{proof}
Fix $\s \in Q^{\mathsf V}$ and $v,w\in \mathsf V$.
By~\eqref{eq:Grimm} it is enough to show that 
\begin{align}  \label{eq:Grimm1}
 \tilde \mu (\s^{v,w}) \tilde \mu (\s_{v,w}) \geq  \tilde \mu(\s^{v}_{w}) \tilde \mu (\s^{w}_v)
\end{align}
By the definition of the spin model $(\s,\s')$, conditioned on~$\s$, the spins~$\s'$ are distributed like two
independent Potts models, both with parameter $J=\beta$, defined on the subgraphs of $\mathsf G$ induced by the two sets of vertices of constant spin $\s$.
Let
\[
V_{i}:=\{ v\in \mathsf V: \s(v)=i\} \setminus \{v,w\}, \qquad i\in \{0,1\},
\]
and let $\mu_{i}$ be the law of the Potts model with $J=\beta$ defined on the subgraph of $\mathsf G$ induced by $V_{i}$.
Define $V_i(v)= \{ v'\in V_{i}:v' \sim v   \}$
and let
\begin{align} \label{eq:Xi}
X_i(v)=\sum_{\s'(v)\in Q'}\prod_{v' \in V_i(v)}e^{\alpha + \beta \delta_{\s'(v),\s'(v')}},
\end{align}
where $\alpha$ and $\beta$ are as in~\eqref{eq:sspin}.

We consider two cases.

{\bf Case I:} $\{v, w\}\notin \mathsf E$. 
One can check that in this case condition \eqref{eq:Grimm1} is equivalent to
\[
\prod_{i\in \{0,1\}}\mu_{i}(X_i(v)X_i(w)) \geq \prod_{i\in \{0,1\}}\ \mu_{i}(X_i(v)) \mu_{i}(X_i(w)),
\]
and it is enough to prove the corresponding inequality for each $i\in\{0,1\}$ separately.
To this end, we write 
\begin{align} \label{eq:linear}
e^{\alpha + \beta \delta_{\s'(v),\s'(v')}}= e^{\alpha} +\delta_{\s'(v),\s'(v')}( e^{ \alpha+\beta }- e^{\alpha}).
\end{align}
Plugging this into~\eqref{eq:Xi} and expanding the product, we have
\begin{align}
X_i(v) & =\sum_{\s'(v)\in Q'} \sum_{A \subseteq V_i(v) } e^{\alpha(|V_i(v)| -|A|) }( e^{ \alpha+\beta }- e^{\alpha})^{|A|}\mathbf 1\{\s'\equiv\s'(v) \textnormal{ on } A \} \nonumber\\
&= \sum_{A \subseteq V_i(v) } e^{\alpha(|V_i(v)| -|A|) }( e^{ \alpha+\beta }- e^{\alpha})^{|A|}\mathbf 1\{\s' \textnormal{ constant on } A \} . \label{eq:expX}
\end{align}
Since $\beta\geq 0$ we have that $e^{ \alpha+\beta }- e^{\alpha} \geq 0$, and hence all the coefficients of the sum are nonnegative.
The desired bound follows therefore from Lemma~\ref{lem:clumping} applied to each term of the sum separately.

{\bf Case II:} $\{v, w\} \in \mathsf E$. 
In this case condition \eqref{eq:Grimm1} can be rewritten as
\[
\prod_{i\in \{0,1\}}\mu_{i}(X_i(v)X_i(w)e^{\alpha + \beta \delta_{\s'(v),\s'(w)}}) \geq \prod_{i\in \{0,1\}}\ \mu_{i}(X_i(v)) \mu_{i}(X_i(w)).
\]
Again, it is enough to prove the corresponding inequality for each $i\in\{0,1\}$ separately. Using \eqref{eq:expX} and \eqref{eq:linear} we can again conclude the desired inequality from
the following bound valid for all $A\subseteq V_i(v)$ and $B\subseteq V_i(w)$:
\begin{align*}
&\mu_{i}\Big(\sum_{\s'(v),\s'(w)\in Q'} \mathbf 1 \{\s' \equiv \s'(v) \textnormal{ on } A \textnormal{ and }\s' \equiv \s'(w) \textnormal{ on } B \}e^{\alpha + \beta \delta_{\s'(v),\s'(w)}}\Big) \\
&= e^{\alpha} \mu_{i}(\s' \textnormal{ constant on } A \textnormal{ and on } B)+( e^{ \alpha+\beta }- e^{\alpha})\mu_{i}(\s' \textnormal{ constant on } A \cup B) \\
& \geq (e^{\alpha} +\tfrac 1{q'} ( e^{ \alpha+\beta }- e^{\alpha}))\mu_{i}(\s' \textnormal{ constant on } A \textnormal{ and on } B) \\
& =\tfrac 1{q'}(e^{ \alpha+\beta }+(q'-1)e^{\alpha})\mu_{i}(\s' \textnormal{ constant on } A \textnormal{ and on } B) \\
& \geq \mu_{i}(\s' \textnormal{ constant on } A)\mu_{i}(\s' \textnormal{ constant on } B),
\end{align*}
where we used Lemma~\ref{lem:clumping} twice, and where we used that $e^{ \alpha+\beta }+(q'-1)e^{\alpha}\geq q'$ for $b\leq 1$.
\end{proof}

\section{Interplay between the models} \label{sec:interplay}
In this section we discuss various aspects of the relationship between the three instances of the model.

\subsection{Spins and percolation}
We start by showing how $\sigma$-spin correlations are described by $\omega$-connectivity probabilities
in the same way the Potts model correlations are given by the random cluster connectivity probabilities.
Using this together with positive association of $\omega$ for $a+b\leq 1$ established in Sect.~\ref{sec:omegaFKG}, we deduce the second Griffiths inequality for the moments of $\sigma$. This is proved in the same manner as the analogous result for the Potts model was proved in~\cite{grimPotts}.

For a random variable $X$, we will sometimes use the physics notation and write $\langle X \rangle$ for the expectation of $X$.
Let $\sigma_0$ be a random variable uniformly distributed on $Q$, and let $m(k)= \langle \sigma_0^k \rangle$ for $k \in \{ 0,1,\ldots\}$.
Since we assume that $Q$ is a symmetric subset of $\mathbb R$, we have $m(2k+1)=0$.
For $\xi\subseteq \mathsf E$ and $v_1,\ldots,v_k \in \mathsf V$, we define $\pi(\xi)$ to be the partition of $\{ v_1,\ldots,v_k \}$ induced by the connected components of~$\xi$. 
\begin{proposition}[$\sigma$-spin correlations via $\omega$-connectivities]
Let $v_1,\ldots,v_k\in \mathsf V$ and $r_1,\ldots,r_k\in \{0,1,\ldots \}$. Then 
\[
\langle \sigma_{v_1}^{r_1} \cdots \sigma_{v_k}^{r_k} \rangle= \sum_{ \mathcal P } \mathbf P (\pi( {\omega})=  \mathcal P) \prod_{A\in \mathcal P} m\Big({\sum_{ v_j\in A} r_{j}}\Big),
\]
where the sum is taken over all partitions $ \mathcal P$ of $\{v_1,\ldots,v_k\}$ such that for all $A\in \mathcal P$, $\sum_{ v_j\in A} r_{j}$ is even.
In particular,
\[
\langle \sigma(v_1)\sigma(v_2) \rangle =  m(2)\mathbf P( v_1 \con{\omega}v_2),
\]
where $\{ v_1 \con{\omega}v_2 \}$ is the event that $v_1$ and $v_2$ are in the same cluster of $\omega$.
\end{proposition}
\begin{proof}
It is enough to condition on $\omega$ and use Proposition~\ref{prop:coupling}.
\end{proof}

\begin{proposition}[Second Griffiths inequality] Assume that $a+b\leq 1$.
Let $v_1,\ldots,v_k\in\mathsf V$ and $r_1,\ldots,r_k,s_1,\ldots,s_k\in \{0,1,\ldots \}$. Then 
\[
\langle \sigma(v_1)^{r_1+s_1} \cdots \sigma(v_k)^{r_k+s_k} \rangle \geq \langle \sigma({v_1})^{r_1} \cdots \sigma({v_k})^{r_k} \rangle\langle \sigma({v_1})^{s_1} \cdots \sigma({v_k})^{s_k} \rangle.
\]
\end{proposition}
\begin{proof}It is not difficult to prove that $m(r+s)\geq m(r)m(s)$ for $r,s\in  \{0,1,\ldots \}$ (see e.g.~\cite{grimPotts}).
Therefore, the function 
\[
\xi \mapsto \prod_{A\in \pi(\xi)} m\Big(\sum_{v_j\in A} r_j \Big)
\]
is increasing on $\Omega$.
Combining this with the proposition above and the positive association of $\omega$ under $\mathbf P$ from Proposition~\ref{prop:fkgomega}, we get
\begin{align*}
\langle \sigma(v_1)^{r_1+s_1} \cdots \sigma(v_k)^{r_k+s_k} \rangle &= \mathbf E \Big[ \prod_{A\in \pi(\omega)} m\Big(\sum_{v_j\in A} r_j +\sum_{v_j\in A} s_{j}\Big)\Big] 
\\ & \geq \mathbf E \Big[ \prod_{A\in \pi(\omega)}    m\Big(\sum_{v_j\in A} r_{j} \Big) m\Big(\sum_{v_j\in A} s_{j}\Big)\Big]  
\\  & \geq \mathbf E \Big[ \prod_{A\in \pi(\omega)}   m\Big(\sum_{v_j\in A} r_{j} \Big)\Big] \mathbf E \Big[ \prod_{A\in \pi(\omega)}   m\Big(\sum_{v_j\in A} s_{j}\Big)\Big]  
\\ &= \langle \sigma({v_1})^{r_1} \cdots \sigma({v_k})^{r_k} \rangle\langle \sigma({v_1})^{s_1} \cdots \sigma({v_k})^{s_k} \rangle. \qedhere
\end{align*} 
\end{proof}

\subsection{Height function and percolation} \label{sec:homega}
We will now compare the variance of the difference of the height function $h'$ between two faces $u_1$ and $u_2$ to the expected number of clusters
of the percolation models that one has to cross when going from $u_1$ to $u_2$. 
We assume here that $M$ is of genus zero but similar arguments taking into account the topology of clusters can be applied to surfaces of higher genus.

Consider the percolation configuration $\omega$, and let $\sigma $ and $\sigma' $ be consistent with~$\omega$, i.e., $\sigma$ is 
constant on the clusters of $\omega$ and $\eta(\sigma') \subseteq \omega$. For a cluster $\mathcal C$ of $\omega$, 
define $\sigma'_{\mathcal C}: \mathsf U \to Q'$ to be the spin configuration which is equal to $\sigma'$ on the faces incident on $\mathcal C$ and satisfies 
$\eta(\sigma'_{\mathcal C}) \subseteq \mathcal C$. In other words, $\sigma'_{\mathcal C}$ is the spin configuration 
obtained from $\sigma'$ by erasing the contours in $\eta(\sigma')$ which are not contained in $\mathcal C$ and 
changing the spins in a consistent way.  Also, denote by $\sigma({\mathcal C})$ the value of $\sigma$ assigned to any vertex in $\mathcal C$.
Fix $u_1,u_2\in \mathsf U$, and let 
\begin{align*} 
dh' &=dh'(u_1,u_2)=h(u_2)-h(u_1), \text{ and } d\sigma'_{\mathcal C}=d\sigma'_{\mathcal C}(u_1,u_2)=\sigma'_{\mathcal C}(u_2 )-  \sigma'_{\mathcal C}(u_1).
\end{align*}
We claim that
\begin{align} \label{eq:hcluster}
dh' =\sum_{\mathcal C} \sigma({\mathcal C}) d\sigma'_{\mathcal C}.
\end{align}
To justify this, we will say that a non-trivial cluster $\mathcal C$ of $\omega$ \emph{disconnects} $u_1$ from~$u_2$, if the two faces belong to two different connected 
components of $M\setminus \mathcal C$, where we think of $\mathcal C$ as the closed subset of $M$ given by the union of its edges.
\begin{remark} \label{rem:ceta}
Note that if $\mathcal C$ does not disconnect $u_1$ from $u_2$, then $d\sigma'_{\mathcal C}=0$. 
Actually if $d\sigma'_{\mathcal C}\neq 0$, then $\mathcal C$ necessarily contains a connected component of $\eta(\sigma')$ that disconnects $u_1$ from $u_2$. 
\end{remark}
By the above remark, 
the sum in \eqref{eq:hcluster} can be restricted to only these clusters that disconnect $u_1$ from~$u_2$.
Let $\gamma =\{ \tilde u_1,\ldots, \tilde u_l\} $ be a path of pairwise adjacent faces with $ \tilde u_1= u_1$ and $\tilde u_l=u_2$, and such that the only 
non-trivial clusters of $\omega$ that it crosses are those that disconnect $u_1$ from~$u_2$, and moreover $\gamma$ crosses each of 
these clusters only once. For $j\in\{1,\ldots,l-1$\}, let $v_j$ be one of the two vertices of the edge dual to $\{\tilde u_j, \tilde u_{j+1} \}$.
By the definition of $h'$ \eqref{eq:proph2} we have
\[
dh' =\sum_{j=1}^{l} \sigma(v_j) (\sigma'(\tilde u_j)-\sigma'(\tilde u_{j+1})).
\]
Note that as long as $\gamma$ stays in between two clusters of $\omega$, the increments in the sum above are zero since the value of $\sigma'(\tilde u_j)$ is constant.
On the other hand, when $\gamma$ crosses a cluster $\mathcal C$, then the value of $\sigma(v_j)$ is constant and equal to $\sigma(\mathcal C)$, and hence, 
by a telescopic sum, the contribution corresponding to $\mathcal C$ is exactly $ \sigma({\mathcal C}) d\sigma'_{\mathcal C}$. This justifies \eqref{eq:hcluster}. 

We will now use this formula to estimate the variance of~$dh'$.
To this end, let $\sigma_0$ be a random variable uniformly distributed on $Q$.
By property \eqref{eq:step1} and~\eqref{eq:step2} from Proposition~\ref{prop:coupling}, 
conditioned on $\omega$, $\sigma(\mathcal C)$ and $\sigma'_{\mathcal C}$ are independent, and moreover $\sigma(\mathcal C)\sim \sigma_0$ and $\sigma(\mathcal C') \sim \sigma_0$ are independent for different clusters $\mathcal C$ and~$\mathcal C'$ of $\omega$. Hence, using \eqref{eq:hcluster} we can write
\begin{align*}
\mathbf{Var}[d h'] =& \mathbf E \Big[ \Big(\sum_{\mathcal C} \sigma({\mathcal C})d\sigma'_{\mathcal C}\Big)^2\Big] \\
= &\sum_{\omega \subseteq \mathsf E } \sum_{\mathcal C_1,  \mathcal C_2 \subseteq \omega}  \mathbf{E}  [  \sigma({\mathcal C_1})d\sigma'_{ \mathcal C_1} \sigma({\mathcal C_2})d\sigma'_{ \mathcal C_2}  \mid\omega]\mathbf P(\omega)  \\
= &\sum_{\omega \subseteq \mathsf E } \sum_{\mathcal C_1,  \mathcal C_2 \subseteq \omega}  \mathbf{E}  [  \sigma({\mathcal C_1}) \sigma({\mathcal C_2}) \mid \omega]\mathbf{E}  [  d\sigma'_{ \mathcal C_1}  d\sigma'_{ \mathcal C_2}  \mid \omega]\mathbf P(\omega)  \\
= & \mathbf{E}  [  \sigma_0^2] \sum_{\omega \subseteq \mathsf E } \sum_{ \mathcal C \subseteq \omega} \mathbf{E}  [  ({d\sigma'_{ \mathcal C}})^2   \mid \omega]\mathbf P(\omega) \\
=&\mathbf{E}  [  \sigma_0^2] \mathbf{E}\Big[ \sum_{\mathcal C} ({d\sigma'_{ \mathcal C}})^2 \Big] \\
=&\mathbf E[\sigma_0^2] \sum_{d\neq 0} d^2 \mathbf E \big[N_d\big],
\end{align*}
where $N_d=N_d({u_1,u_2})$ is the number of clusters $\mathcal C$ of $\omega$ such that $d\sigma'_{\mathcal C}=d$.
Define $N_{\neq 0}= \sum_{d\neq 0}N_d$, and let $C/2=\max \{| i|: i\in Q'\}$. Then $|d\sigma'_{\mathcal C}| \leq C$, and from the above computation we immediately get the following equivalence up to constants.
\begin{proposition} \label{prop:uptoconst}
We have
\begin{align*} 
C^2 \mathbf E[\sigma_0^2]  \mathbf E[N_{\neq 0} ]\geq \mathbf{Var}[dh']  \geq \mathbf E[\sigma_0^2]  \mathbf E\big[ N_{\neq 0}\big ].
\end{align*}
\end{proposition}
Note that in the special case when $q'=2$ and $Q'=\{ -i,i\}$, we actually get the equality
\begin{align*}
 \mathbf{Var}[dh']  = 4i^2 \mathbf E[\sigma_0^2] \mathbf E\big[ N_{\neq 0}\big ].
\end{align*}
We note that this identity in the setting of the double random current model and the related height function was first obtained in~\cite{DumCopLis}, where it was
used to establish continuity of the phase transition of the Ising model on any bi-periodic planar graph.

To prove the final result of this section we will need the following description of the mutual conditional laws for the two bond percolation processes.
\begin{lemma} \label{lem:condomega}
Assume that $a+b \geq1$. Then conditioned on $(\omega', \sigma')$, the percolation configuration $\omega$ is distributed like $\eta(\sigma') \cup \zeta$, where $\zeta$ is an independent bond percolation process on $(\omega')^{\dagger}$ with success probability $(1-b)/a$.
\end{lemma}

\begin{proof}
This follows from the fact that, tossing the four-sided die from the right-hand side of the table in~\eqref{step1}, and conditioning on $e^*$ being closed,
we open $e$ with probability $(1-b)/a$.
\end{proof}

We say that a cluster $\mathcal C'$ of $\omega'$ \emph{disconnects} face $u_1$ from $u_2$, if either one of the two faces is in $\mathcal C'$, or they belong to two different connected 
components of $M\setminus \mathcal C'$, where we think of $\mathcal C'$ as the closed subset of $M$ given by the union of its edges.
Note that the trivial clusters containing $u_1$ and $u_2$ satisfy this definition.

The next result compares the number of clusters of $\omega'$ disconnecting $u_1$ from $u_2$ with the number of clusters $\mathcal C$ of $\omega$ which satisfy $d \sigma'_{\mathcal C} \neq 0$ in the case $a+b\geq 1$.
\begin{proposition} \label{prop:hfclusters}
Let $N'=N'(u_1,u_2)$ be the number of clusters of $\omega'$, that \emph{disconnect} $u_1$ from $u_2$. 
Assume that $a+b\geq 1$. Then 
\[
\mathbf E[ N_{\neq 0} ] \geq (1-\tfrac{1}{q'}) (\mathbf E[ N']-1).
\]
\end{proposition}
\begin{proof} 
Recall that by Proposition~\ref{prop:coupling}, when conditioned on $\omega'$, the spins $\sigma'$ are assigned  to each cluster of
$\omega'$ independently and uniformly in $Q'$. Let $\mathcal C'_1,\ldots,\mathcal C'_{N'}$ be the clusters of $\omega'$ that disconnect $u_1$ from $u_2$ ordered
according to the first intersection points with a chosen path from $u_1$ to $u_2$. Note that if two consecutive clusters $\mathcal C'_l, \mathcal C'_{l+1}$ are assigned different spins, then for topological reasons, there must exist
a circuit in $\eta(\sigma')$ disconnecting $\mathcal C'_i$ from $\mathcal C'_{l+1}$, and hence also disconnecting $u_1$ from $u_2$.

By Lemma~\ref{lem:condomega}, conditioned on $\sigma'$ and $\omega'$, we recover $\omega$ by choosing randomly edges from $(\omega')^{\dagger}$ and adding them to $\eta(\sigma')$.
This means that for every pair $\mathcal C'_l, \mathcal C'_{l+1}$ with different spin $\sigma'$, there exists at least one cluster $\mathcal C$ of $\omega$, disconnecting $u_1$ from $u_2$.
Moreover, at least one of these clusters must satisfy $d\sigma'_{\mathcal C}(u_1,u_2)\neq 0$ (since the sum of $d\sigma'_{\mathcal C}$ over all such clusters is nonzero).
Since the clusters of $\omega'$ are disjoint, the clusters of $\omega$ corresponding to different pairs $\mathcal C'_l, \mathcal C'_{l+1}$ are also disjoint.
This means that $N_{\neq 0}$ is at least equal to the number of pairs $\mathcal C'_l, \mathcal C'_{l+1}$ with different spin $\sigma'$.
The latter is equal in distribution to the number of nearest neighbour disagreements in an i.i.d.\ sequence of length $N'$ and distribution $\sigma_0$. Hence, we get the desired inequality
by an elementary computation of the expectation.
\end{proof}

\section{The model on $\mathbb Z^2$} \label{sec:square}
 
We now apply these relatively abstract results to study the behaviour of the model on $\mathbb Z^2$. We consider three different regimes of parameters:
\begin{itemize} 
\item the line $a+b=1$ corresponding to the random cluster model, 
\item the case when either $a$ or $b$ is small,
\item the self-dual case with $q=q'$ and $a=b\geq 1/2$. 
\end{itemize}
We show localization in the first two cases, except at the point $q=q'=2$, $a=b=1/2$ where delocalization is established. This was first proved in~\cite{GlaPel} for different boundary conditions (which are arguably more natural for the corresponding six-vertex model).
We also present some partial results towards establishing delocalization in the third case.

To this end, let $\Lambda_n$ be the $2n\times 2n$ box in $\mathbb Z^2$ centered around the origin, and let $\mathbf P_{\Lambda_n}$ be the law of the model on $\Lambda_n$.
The fact that $\Lambda_n$ is a concretely defined box will actually not be relevant for the arguments in the first two cases (as long as it is monotonically growing with $n$) and we choose it only for concreteness.
\subsection*{The random cluster model line $a+b=1$}
By Proposition~\ref{prop:RC} and Remark~\ref{rem:fkcrit}, for $a+b=1$, the marginal of $\mathbf P_{\Lambda_n}$ on $\omega$ is the $\textnormal{FK}(qq')$ random cluster model on $\Lambda_n$ with free boundary conditions and with parameter 
\[
p({a}) =\frac{q'}{q'+a^{-1}-1}.
\]
We recall that the critical value of $p$ for this model was rigorously established in~\cite{BefDum} and is equal to 
\begin{align} \label{eq:BefDum}
p_c(qq')= \frac{\sqrt{qq'}}{1+\sqrt{qq'}}.
\end{align}
In what follows we will call a random cluster model (possibly defined on a finite graph) \emph{subcritical}, resp.\ \emph{supercritical}, if the associated parameter $p$ is strictly smaller, resp.\ larger than $p_c$.

We also recall that the laws of the random cluster model on $\Lambda_n$ with free boundary conditions are stochastically increasing (by strong positive association and domain Markov property of the random cluster model,~see e.g.\ \cite{RC}), and hence the weak limit $\mathbf P_{\mathbb Z^2} =\lim_{n\to \infty} \mathbf P_{\Lambda_n}$ 
exists, and is the infinite-volume random cluster model with free boundary conditions (here, with a slight abuse of notation, we do not distinguish between the full 
distribution of the model and its marginal distribution on~$\omega$).

To study the height function we use the recent results rigorously establishing the order of phase transition in the random cluster model~\cite{DCST,Disc,SpiRay1}, and we conclude the following two different behaviours. 
\begin{theorem} \label{cor:twobeh}Consider $\mathbf P_{\Lambda_n}$ for $a+b=1$. Let $h'_n$ be the height function on the faces of $\Lambda_n$ with height zero 
assigned to the external face $u_{\infty}$, and let $u_0$ be a fixed face next to the origin of $\mathbb Z^2$. If $q=q'=2$ and $a=b=1/2$, then
\begin{align*}
  \mathbf{Var}_{\Lambda_n} [h'_n(u_0)] \to \infty
  \end{align*}
as $n \to \infty$, and otherwise there exists $C=C(q,a) < \infty$ such that 
\begin{align*} 
  \mathbf{Var}_{\Lambda_n} [h'_n(u_0)] \leq C
\end{align*}
for all $n$.
\end{theorem}
\begin{proof}
To prove the first statement, note that since $p(1/2)=2/3=p_c(4)$, $\mathbf P_{\mathbb Z^2}$ is the critical $\textnormal{FK}(4)$ random cluster model on $\mathbb Z^2$ with free boundary conditions.
In particular, by~\cite{DCST} we know that $\mathbf P_{\mathbb Z^2}$-almost surely, there are infinitely many clusters of $\omega$, and hence also of $\omega'=\omega^{\dagger}$, surrounding $u_0$, i.e.,
disconnecting $u_0$ from infinity.
Let $N'(k)$ be the number of clusters of $\omega'$ that {surround} $u_0$ and that are contained in the box $\Lambda_k$. Combining the lower 
bounds from Propositions~\ref{prop:hfclusters} and~\ref{prop:uptoconst} applied to $u_0$ and $u_{\infty}$, we obtain that there exists an (explicit) constant $C'>0$ such that for all $k\in\{1,2,\ldots\}$,
\[
\liminf_{n\to \infty}\mathbf{Var}_{\Lambda_n} [h'_n(u_0)]  \geq C' \liminf_{n\to \infty}( \mathbf E_{\Lambda_n}[N'(k)]  -1)= C' ( \mathbf E_{\mathbb Z^2}[N'(k)]  -1),
\]
where the identity follows from the weak convergence of $\mathbf P_{\Lambda_n}$ and the fact that $N'(k)$ is a bounded local random variable. 
It is now enough to notice that since $N'(k)$ is increasing in $k$, $\mathbf E_{\mathbb Z^2}[N'(k)] \to \infty$ as $k\to \infty$ by monotone convergence.

We now turn to the second statement and note that if either $q=1$ or $q'=1$, then the height function is trivially uniformly bounded almost surely. We will now use a variant of Peierls' argument to show that the number of clusters of $\omega$ surrounding $u_0$ is uniformly bounded in expectation. 
We consider three cases: 

{\bf Case I}: $p(a) < p_c$.
In this case, $\mathbf P_{\mathbb Z^2}$ is a subcritical random cluster model on $\mathbb Z^2$ with free boundary conditions.
By sharpness of the phase transition~\cite{BefDum}, this means that the $\mathbf P_{\mathbb Z^2}$-probability that two vertices belong to the same cluster of~$\omega$ decays exponentially fast with the distance between the vertices.
Let $N$ be the number of clusters of $\omega$ surrounding~$u_0$. We have
\begin{align}
\mathbf E_{\Lambda_n}[ N] & \leq \sum_{k=0}^n \mathbf P_{\Lambda_n}[\omega\textnormal{-cluster containing }(-k,-k) \textnormal{ surrounds } u_0] \nonumber \\
&\leq \sum_{k,l=0}^n \mathbf P_{\Lambda_n}[(-k,-k) \con{\omega} (l,l)] \nonumber \\
&  \leq  \sum_{k,l=0}^{\infty} \mathbf P_{\mathbb Z^2}[(-k,-k) \con{\omega} (l,l)] \nonumber \\
&\leq C, \label{eq:peierls}
\end{align}
with $C<\infty$ independent of $n$, where in the second last inequality we used monotonicity of connectivity probabilities with respect to the increasing boundary conditions, and
in the last one we used the fact that the connectivity probabilities decay exponentially for $\mathbf P_{\mathbb Z^2}$. 
We finish the proof by combining this inequality with the upper bound from Propositions~\ref{prop:hfclusters} applied to $u_0$ and $u_{\infty}$, and the obvious bound $\mathbf E_{\Lambda_n}[ N_{\neq 0}] \leq  \mathbf E_{\Lambda_n}[ N]$.

{\bf Case II}: $p(a) = p_c$. In this case, $\mathbf P_{\mathbb Z^2}$ is the critical $\textnormal{FK}(qq')$ random cluster model on $\mathbb Z^2$ with free boundary conditions and $qq'>4$.
In particular, by Theorem 1.2 of~\cite{Disc} establishing discontinuity of phase transition (see also~\cite{SpiRay1} for a more elementary proof), the $\mathbf P_{\mathbb Z^2}$-probability that two vertices belong to the same cluster of $\omega$ again decays exponentially fast and we can apply exactly the same arguments as in~\eqref{eq:peierls}. 

{\bf Case III}: $p(a) > p_c$. Note that by duality in the random cluster model (see e.g.\ \cite{RC}), the law of $\omega'=\omega^{\dagger}$ under $\mathbf P_{\Lambda_n}$, is again a subcritical random cluster model on the dual graph $\Lambda_n^*$. To finish the proof, one has to adjust the arguments in~\eqref{eq:peierls} 
since now the boundary conditions are stochastically \emph{decreasing} for $\omega'$. This issue can be overcome by conditioning on the cluster of~$u_{\infty}$. 

To this end, let $N'$ be the number of clusters of $\omega'$ which surround or contain~$u_0$. Since $\omega'=\omega^{\dagger}$, for topological reasons we have that $N\leq N'+1$, where $N$ is as in Case I.
To finish the proof, it is therefore enough to uniformly bound the expectation of $N'$.  Let $\mathcal C'_{\infty}$ be the cluster of $\omega'$ containing $u_{\infty}$, and 
let $\overline{\mathcal C'}_{\infty}$ be this cluster enlarged by all the edges of $\Lambda_n^{*}$ that are incident on $\mathcal C'_{\infty}$.
By the Markov property of the random cluster model, conditioned on $\mathcal C'_{\infty}$, the rest of the configuration $\omega'$ is distributed like the random cluster model
on $\Lambda_n^{*}$ with the edges in $\overline{\mathcal C'}_{\infty}$ removed. Note that this graph is a proper subgraph of $(\mathbb Z^2)^*$ (unlike $\Lambda_n^*$ where $u_{\infty}$ is of very high degree), 
and hence the law of of the rest of the configuration $\omega'$ is stochastically dominated by the law of $\omega^{\dagger}$ under $\mathbf P_{\mathbb Z^2}$. Since $\mathcal C'_{\infty}$ contributes
at most one to the count of clusters surrounding~$u_0$, we can now repeat the arguments as in~\eqref{eq:peierls} to finish the proof. 
\end{proof}

\subsection*{The case of small $a$ or $b$}
To treat this case, we first derive a simple general comparison result between $\omega$ and the $\textnormal{FK}(q)$ random cluster model. 
Note the difference with Proposition~\ref{prop:RC}, which is valid only for $a+b=1$, and in which~$\omega$ is shown to be a $\textnormal{FK}(qq')$ random cluster model.
\begin{lemma} \label{lem:comparison}
Assume that $M$ is of genus zero. Then the process $\omega$ stochastically dominates the $\textnormal{FK}(q)$ random cluster model on $\mathsf G$ with parameter $p=1-b$.
By duality, the analogous statement holds for $\omega'$ with $\mathsf G$ replaced by $\mathsf G^*$, and $b$ replaced by $a$.
\end{lemma}
\begin{proof}
By~\eqref{eq:od1}, for fixed $ \sigma' \in Q'^{\mathsf U}$, we have for all $\omega$ with $\eta(\sigma')\subseteq \omega$,
\begin{align} \label{eq:again}
\mathbf P (\omega,\sigma')& \propto  \big(\tfrac{a}{1-b}\big)^{|\eta(\sigma')|} q^{k(\omega)}  (1-b)^{|\omega|} b^{|\mathsf E\setminus \omega| } .
\end{align}
This means that conditioned on $\sigma'$, $\omega$ is 
distributed as the $\textnormal{FK}(q)$ random cluster model on $\mathsf G$ conditioned on the event that all edges in $\eta(\sigma')$ are open.
By strong positive association of the random cluster model (see e.g.~\cite{RC}), this law stochastically dominates the corresponding random cluster model on $\mathsf G$ with no restriction on the state of the edges in $\eta(\sigma')$.
Since this holds for every~$\sigma'$, the proof is completed.
\end{proof}

\begin{corollary} Consider $\mathbf P_{\Lambda_n}$ with either $a<(\sqrt{q'}+1)^{-1}$ or $b<(\sqrt{q}+1)^{-1}$. 
Let $h'_n$ be the height function on the faces of $\Lambda_n$ with height zero 
assigned to the external face $u_{\infty}$, and let $u_0$ be a fixed face next to the origin of $\mathbb Z^2$.
Then, there exists $C=C(q,q',a,b) < \infty$ such that 
\begin{align} 
  \mathbf{Var}_{\Lambda_n} [h'_n(u_0)] \leq C
\end{align}
for all $n$.
\end{corollary}
\begin{proof}
{\bf Case I}: $a<(\sqrt{q'}+1)^{-1}$.
By Lemma~\ref{lem:comparison}, the explicit value of the critical parameter~\eqref{eq:BefDum}, and duality in the random cluster model, $(\omega')^{\dagger}$ is stochastically dominated by a 
subcritical $\textnormal{FK}(q')$ random cluster model on $\Lambda_n$. In particular the connectivity probabilities of the former process are bounded above by those of the latter.
To finish the proof in this case it is therefore enough to repeat the reasoning as in \eqref{eq:peierls}.

{\bf Case II}: $b<(\sqrt{q}+1)^{-1}$. 
Recall the upper bound from Proposition~\ref{prop:uptoconst}. 
By Remark~\ref{rem:ceta}, every cluster $\mathcal C$ of $\omega$ which satisfies $d\sigma'_{\mathcal C}(u_0,u_{\infty})\neq 0$ (and therefore contributes to $N_{\neq 0}(u_0,u_\infty)$) must necessarily
contain a connected component of $\eta(\sigma')$ that surrounds $u_0$. 
We can therefore assume that there exists an outermost cycle of edges in $\eta(\sigma')$, denoted by $\eta_0$, that surrounds $u_0$. We define $\Lambda_n^{\eta_0}$ to be the subgraph of 
$\Lambda_n$ contained within $\eta_0$ (including $\eta_0$) with all vertices belonging to~$\eta_0$ identified with one another. 
To conclude the proof it is hence enough to uniformly bound the expected number of clusters of $\omega$ that surround $u_0$ and that are contained in $\Lambda_n^{\eta_0}$.
Again, for topological reasons, this is the same as bounding the expected number of clusters that surround $u_0$ of the dual configuration~$\omega^{\dagger}$ restricted to the dual graph~$(\Lambda_n^{\eta_0})^*$.

To this end, note that by formula \eqref{eq:again} and arguments as in the proof of Lemma~\ref{lem:comparison}, conditioned on $\eta_0$ and the configuration $(\omega, \sigma')$ outside~$\eta_0$, the rest of the configuration $\omega$
stochastically dominates a supercritical random cluster model on $\Lambda_n^{\eta_0}$. Therefore, by duality in the random cluster model, $\omega^{\dagger}$ restricted to $(\Lambda_n^{\eta_0})^*$ is stochastically dominated 
by a subcritical random cluster model on $(\Lambda_n^{\eta_0})^*$. Since $(\Lambda_n^{\eta_0})^*$ is a proper subgraph of the dual square lattice $(\mathbb Z^2)^*$, we can again use the reasoning from \eqref{eq:peierls} to finish the proof.
\end{proof}

\subsection*{The self-dual model with $q=q'$ and $a=b$}
For $q=q'=2$, the question of delocalization at $a=b> 1/2$ is still open for a large range of $a$. 
One of the difficulties in studying this regime is that, e.g.\ the classical Baxter--Kelland--Wu coupling~\cite{BKW} between the six-vertex model
and the random cluster model is no longer a probabilistic construction but involves complex-valued measures.
In our case this is the regime where the marginal of $\mathbf P$ on $\omega$ is not positively associated which is a major technical obstacle. 
We say that a bond percolation process on an infinite graph \emph{percolates}
if it contains an infinite cluster.
The main contribution of this section is a partial result saying that no percolation of $\omega$ is a sufficient condition for delocalization in the self-dual model.

To make full use of translation invariance, we will consider the model ${\mathbf  P}_{\mathbb T_n}$ defined on the square lattice torus $\mathbb T_n$ of size $n\times n$.
Self-duality implies that $\omega$ shifted by $(\tfrac 12,\tfrac 12)$ (so that $\mathbb T_n$ becomes $\mathbb T^*_n  \simeq \mathbb T_n$) has the 
same distribution under $\mathbf P_{\mathbb T_n}$ as~$\omega'$. This property clearly carries over to any subsequential limit $\mathbf P_{\mathbb Z^2} =\lim_{k\to \infty}{\mathbf  P}_{\mathbb T_{n_k}}$.
We talk about subsequential limits here since there is no stochastic monotonicity of the model that would guarantee the uniqueness of the limit.
In our last theorem we show that if $\omega$ does not percolate $\mathbf P_{\mathbb Z^2}$-almost surely, 
then $\omega$ necessarily contains infinitely many clusters surrounding the origin and hence the associated height function delocalizes.
The first implication is immediate by self-duality whenever $\omega^{\dagger}\subseteq\omega'$, but this property only holds for $a=b\leq1/2$. To get the result in full generality 
we use the essential fact that $\eta(\sigma)\subseteq \omega'$, which in particular implies that if $\eta(\sigma)$ percolates, then so does~$ \omega'$. 

\begin{theorem} \label{thm:selfdual} Consider a subsequential limit $\mathbf P_{\mathbb Z^2} =\lim_{k\to \infty}{\mathbf  P}_{\mathbb T_{n_k}}$ of the self-dual model with $q=q'$ and $a=b>1/2$, and
assume that 
\begin{align} \label{cond:perco}
\mathbf P_{\mathbb Z^2} (\omega \textnormal{ percolates}) =0.
\end{align}
Then 
\begin{align} \label{eq:infinite}
\mathbf P_{\mathbb Z^2} (\textnormal{infinitely many clusters of }\omega \textnormal{ surround the origin}) =1.
\end{align}
and
\begin{align} \label{eq:divh}
\lim_{ \ |u_1-u_2|\to \infty } \lim_{k\to \infty} \mathbf{Var}_{\mathbb T_{n_k}} [h'(u_1)-h'(u_2)] = \infty,
\end{align}
where the height increment $h'(u_1)-h'(u_2)$ is computed along one of the shortest paths from $u_1$ to $u_2$ in the dual torus $\mathbb T_{n_k}^*$.
\end{theorem}
We note that the particular choice of the path in the statement above is not essential to the argument.

\begin{remark}
Since $a=b> \tfrac 12$, we have that
$\omega^* \cap \omega' =\emptyset$. 
One therefore expects \eqref{cond:perco} to hold true if e.g.\ one can establish ergodicity of the marginal of $\mathbf P_{\mathbb Z^2}$ onto~$\omega$.
Indeed, if this were true, then by duality, positive probability of percolation would imply coexistence of disjoint infinite clusters of $\omega$ and $\omega'$, which one does not expect to happen.
However, it is not clear why the infinite volume limit should be ergodic, and indeed it follows from the results of~\cite{Disc} and Proposition~\ref{prop:RCtorus} that this is not true in the boundary case $a=1/2$ and $q>2$.
\end{remark}

To prove the theorem, we will need the following lemma. It is highly likely that this result exists in the literature but we could not find a proper reference, and thus we give a proof for completeness.
\begin{lemma} \label{lem:coin}
Condition \eqref{cond:perco} guarantees that property \eqref{eq:step1} from 
Proposition~\ref{prop:coupling}, which says that under $\mathbf P_{\mathbb T_{n_k}}$ the spins $\sigma$ are sampled by independently chosing a spin for each cluster of $\omega$, carries over into the infinite volume limit $\mathbf P_{\mathbb Z^2}$.
\end{lemma}
\begin{proof}
Let $\mathbf S$ be the law on spin configurations on the vertices of $\mathbb Z^2$ coupled with $\mathbf P_{\mathbb Z^2}$ by independently assigning a spin from $Q$ to each cluster of $\omega$.
We need to prove that $\mathbf S$ is equal in distribution to the marginal of $\mathbf P_{\mathbb Z^2}$ on $\sigma$.
To this end, consider a fixed box $\Lambda_l$ and $\varepsilon>0$. 
For $L>l$, let $E_{l,L}$ be the event that \emph{no} cluster of $\omega$ connects $\Lambda_l$ to the boundary of $\Lambda_L$, 
and take $L$ so large that $\mathbf P_{\mathbb Z^2}(E_{l,L})\geq 1-\varepsilon$.
Such $L$ exists since by \eqref{cond:perco} there are only finite clusters almost surely.
Next, take $k$ so large that $n_k>L$ and the total variation distance between the law of $\mathbf P_{\mathbb T_{n_{k}}}$ and $\mathbf P_{\mathbb Z^2}$ restricted to $\Lambda_L$ is 
smaller than $\varepsilon$.
This is possible by weak convergence and since $\Lambda_L$ is fixed. Let $\mathbf P_1$ be the former and $\mathbf P_2$ the latter law, and denote by $(\omega_1,\sigma_1)$ and 
$(\omega_2,\sigma_2)$ the corresponding configurations restricted to $\Lambda_L$. By the classical property of the total variation distance, there exists a coupling $\mathbf Q$ of $\mathbf P_1$ and $\mathbf P_2$ satisfying 
$\mathbf Q(\omega_1= \omega_2, \sigma_1= \sigma_2)\geq 1-\varepsilon$, and hence also $\mathbf Q(\omega_1= \omega_2, \sigma_1= \sigma_2,E_{l,L} )\geq 1-2\varepsilon$.
Note that conditioned on the latter event, by property \eqref{eq:step1} from 
Proposition~\ref{prop:coupling}, $\mathbf S$ and $\sigma_1$ have the same law when restricted to $\Lambda_l$.
The last inequality therefore implies that the total variation distance between $\mathbf S$ and the marginal of $\mathbf P_{\mathbb Z^2}$ on $\sigma$ restricted to $\Lambda_l$ is smaller 
than $2\varepsilon$. Since $l$ and $\varepsilon$ were arbitrary, this ends the proof.
\end{proof}
We are now ready to prove our last theorem.
\begin{proof}[Proof of Theorem~\ref{thm:selfdual}]

For $i\in Q$, let $\textnormal{Perc}(i)$ be the event that the subgraph of $\mathbb Z^2$ induced by $\{ v: \sigma(v)=i\}$ contains an infinite connected component.
By Lemma~\ref{lem:coin} we know that $\sigma$ is distributed like an independent assignment of a spin to each cluster of $\omega$.
By \eqref{cond:perco} all clusters of $\omega$ are finite almost surely, and hence $\textnormal{Perc}(i)$ are tail events with respect to these independent spin assignments. 
Therefore by Kolomogorov's 0-1 law, conditioned on $\omega$, the probability of $\textnormal{Perc}(i)$ is either $0$ or $1$.
In the latter case, by symmetry under spin relabelling, we have that the probability of $\bigcap_{i\in Q}\textnormal{Perc}(i)$ is also $1$. Hence,
\begin{align}
\mathbf{P}_{\mathbb Z^2} \Big(\bigcup_{i\in Q}\textnormal{Perc}(i)\Big )&= \mathbf{E}_{\mathbb Z^2}\Big[\mathbf P_{\mathbb Z^2}\Big(\bigcup_{i\in Q}\textnormal{Perc}(i) \mid \omega \Big)\Big]\nonumber \\
&=  \mathbf{E}_{\mathbb Z^2}\Big[\mathbf P_{\mathbb Z^2}\Big(\bigcap_{i\in Q}\textnormal{Perc}(i)\mid \omega \Big)\Big]\nonumber \\
&= \mathbf{P}_{\mathbb Z^2} \Big(\bigcap_{i\in Q}\textnormal{Perc}(i)\Big ).\label{eq:perc}
\end{align}
Note that for topological reasons, on the event $\textnormal{Perc}(i) \cap \textnormal{Perc}(j)$ for $i\neq j$, there must be at least one infinite interface in $\eta(\sigma)$ 
separating two infinite components with different $\sigma$-spins. Since $\eta(\sigma)\subseteq \omega'$, we infer that on the event $\textnormal{Perc}(i) \cap \textnormal{Perc}(j)$, 
the configuration $\omega'$ percolates. By self-duality and~\eqref{cond:perco}, $\omega'$ does not percolate a.s., and therefore $\mathbf P_{\mathbb Z^2}(\textnormal{Perc}(i) \cap \textnormal{Perc}(j))=0$ 
for $i\neq j$. Hence by \eqref{eq:perc}, $\mathbf P_{\mathbb Z^2}(\textnormal{Perc}(i))=0$ for all $i\in Q$. 
We therefore conclude that there are infinitely many clusters of $\eta(\sigma)$, and hence also of $\eta(\sigma')$, surrounding the origin a.s. Since $\omega$ does not percolate and $\eta(\sigma')\subseteq \omega$ a.s., 
there must be infinitely many clusters of $\omega$  surrounding the origin a.s. which gives \eqref{eq:infinite}.

It is now enough to use~\eqref{eq:infinite} to deduce delocalization of the height function~\eqref{eq:divh}. The argument is analogous to the proof of the first part of Theorem~\ref{cor:twobeh}
but the required results from Sect.~\ref{sec:homega} need to be adjusted to the topology of the torus by also considering noncontractible clusters that may intersect the path along which the increment of the height function is computed. We leave the details to the reader.
\end{proof}

\bibliographystyle{amsplain}
\bibliography{iloop}
\end{document}